\documentclass[11pt,a4paper]{article}

\usepackage{epsf,epsfig,amsfonts,amsgen,amsmath,amstext,amsbsy,amsopn,amsthm
}
\usepackage{amsmath,amssymb}
\usepackage{graphicx}
 \usepackage{tikz}
\usepackage{caption}
\usepackage{amsmath,times,mathptmx}
\usepackage{enumitem}
\usepackage{amsfonts,amsthm,amssymb}
\usepackage{amsfonts}
\usepackage{graphics}
\usepackage{latexsym,bm}
\usepackage{amsfonts,amsthm,amssymb,bbding}
\usepackage{indentfirst}
\usepackage{graphicx}
\usepackage{color}
\usepackage[colorlinks=true,anchorcolor=blue,filecolor=blue,linkcolor=blue,urlcolor=blue,citecolor=blue]{hyperref}
\usepackage{float}
\usepackage{tikz}
\usepackage{amsthm}
\usepackage{empheq}
\usepackage{cite}
\DeclareMathAlphabet{\mathcal}{OMS}{cmsy}{m}{n}
\DeclareSymbolFont{largesymbols}{OMX}{cmex}{m}{n}
\allowdisplaybreaks[4]
\usepackage{moresize}

\newtheorem{thm}{Theorem}

\newtheorem{lemma}{Lemma}
\newtheorem{false statement}{False statement}

\theoremstyle{definition}

\newtheorem{claim}{Claim}

\theoremstyle{plain}
\theoremstyle{definition}
\newtheorem{remark}[lemma]{Remark}

\newtheorem{case}{Case}
\newtheorem{casee}{Case}

\newtheorem{subcase}{Subcase}[case]
\newtheorem{subcasee}{Subcase}[casee]

\newtheorem{problem}{Problem}

\newcommand{\R}{\mathbb R}
\newcommand{\bibanchor}[1]{\hypertarget{ref:#1}{}}

\DeclareSymbolFont{CfourSigma}{OML}{cmm}{m}{it}
\SetSymbolFont{CfourSigma}{bold}{OML}{cmm}{b}{it}
\DeclareMathSymbol{\sigma}{\mathord}{CfourSigma}{"1B}

\begin{document}
	
\title{On the maximal index of $C_4^-$-free unbalanced signed graphs of prescribed size 
	\footnote{Supported by Tianshan Talent Training Program (No. 2024TSYCCX0013),
		Natural Science Foundation of Xinjiang Uygur Autonomous Region (No. 2024D01C41),
		The Basic scientific research in universities of Xinjiang
		Uygur Autonomous Region (XJEDU2025P001) and NSFC (No. 12361071).}}
\author{{Zhihui Yang, Dan Li\thanks{Corresponding author. E-mail: ldxjedu@163.com.}, Yuanyuan Chen}\\
{\footnotesize College of Mathematics and Systems Science, Xinjiang University, Urumqi 830046, China}}
\date{}

\maketitle {\flushleft\large\bf Abstract:}
The extremal problem of determining the graphs with maximum spectral radius among all graphs with a prescribed number of edges was resolved by Rowlinson
[Linear Algebra Appl. 110 (1988) 43--53]. A signed graph $\Gamma=(G,\sigma)$ is called $C_k^-$-free if it contains no negative $k$-cycle as a signed subgraph. The index of a signed graph is defined as the largest eigenvalue of its adjacency matrix. This paper addresses the maximization of the index of unbalanced signed graphs.
For signed graphs of a prescribed order, Wang and Lin [Discrete Appl. Math. 372 (2025) 164--172] determined the maximum index of $C_4^{-}$-free unbalanced signed graphs. Subsequently, Wang, Peng and Hou [Discrete Math. 349 (2026) 114871] characterized the unbalanced signed graphs of prescribed size with maximum index. This paper treats the edge spectral Turán-type problem for $C_4^-$-free unbalanced signed graphs. The all-positive signed graph with underlying graph $G$ is denoted by $(G,+)$. Evidently, $A((G,+))=A(G)$.
In particular, $(H_{d,t},+)$ is the all-positive signed graph whose underlying graph is $H_{d,t}$, where $H_{d,0}=K_d$ and for $t>0$, $H_{d,t}$ is obtained from $K_d$ by adding a single vertex adjacent to exactly $t$ vertices of $K_d$. Let $m-3=\binom{d}{2}+t$ with $0\le t\le d-1$. 
Up to switching isomorphism and the addition of isolated vertices, the unique extremal unbalanced signed graph is obtained by attaching a negative triangle to a specified vertex of the all-positive signed graph $(H_{d,t},+)$.

\vspace{0.1cm}
\begin{flushleft}
\end{flushleft}

\section{Introduction}

Throughout this paper, we consider finite, undirected, simple graphs. Let $G=(V(G),E(G))$ be a graph with vertex set $V(G)=\{v_1,\ldots,v_n\}$ and edge set $E(G)$. The order and size of $G$ are $n=|V(G)|$ and $m=|E(G)|$, respectively. The neighborhood of a vertex $v$ is denoted by $N_G(v)$, and its degree is $d_G(v)=|N_G(v)|$. For convenience, we write $d(v)=d_G(v)$. For a subset $S\subseteq V(G)$, the induced subgraph $G[S]$ has vertex set $S$ and contains precisely the edges of $G$ with both endpoints in $S$.
 A signed graph $\Gamma=(G,\sigma)$ consists of an underlying graph $G$ and a sign function $\sigma:E(G)\to\{+1,-1\}$. An edge $e$ is called positive if $\sigma(e)=+1$ and negative if $\sigma(e)=-1$. A cycle in $\Gamma$ is positive if it contains an even number of negative edges, otherwise it is negative. A signed graph $\Gamma=(G,\sigma)$ is balanced if it contains no negative cycles, and unbalanced otherwise. Let $U \subseteq V(G)$. The operation that changes the signs of all edges between $U$ and $V(G)\setminus U$ is called a switching operation. If a signed graph $\Gamma'$ is obtained from $\Gamma$ by a sequence of switching operations, then $\Gamma$ and $\Gamma'$ are switching equivalent. Two signed graphs $\Gamma$ and $\Gamma'$ are switching isomorphic if there exists a vertex subset $U\subseteq V(\Gamma)$ such that the signed graph obtained by switching $\Gamma$ with respect to $U$ is isomorphic to $\Gamma'$.
The adjacency matrix of $\Gamma$ is $A(\Gamma)=(a_{uv})$, where $a_{uv}=\sigma(uv)$ if $uv\in E(G)$ and $a_{uv}=0$ otherwise. Write $\lambda_1(A(\Gamma))\ge\cdots\ge\lambda_n(A(\Gamma))$ for the eigenvalues of $A(\Gamma)$ in nonincreasing order. The largest eigenvalue $\lambda_1(A(\Gamma))$, abbreviated as $\lambda_1(\Gamma)$, is called the index of $\Gamma$. The spectral radius of $G$ is $\rho(G)=\max_i|\lambda_i(G)|$. Degrees and neighborhoods in a signed graph refer to its underlying graph, and $\Delta$ denotes the maximum degree. The basic theory of signed graphs goes back to Harary and Zaslavsky \cite{Harary1953,Zaslavsky1982,Zaslavsky2010}, standard references for graph spectra and real symmetric matrices are \cite{BrouwerHaemers,HornJohnson}. 

Let $\mathcal{F}$ be a family of graphs. A graph $G$ is called $\mathcal{F}$-free if $G$ contains no member of $\mathcal{F}$ as a subgraph. In particular, if $\mathcal{F} = \{F\}$, then the graph is called $F$-free. The classical spectral Turán problem is to determine the maximum spectral radius of an $\mathcal{F}$-free graph of order $n$, and this maximum value is also known as the spectral Turán number of $\mathcal{F}$.
 The edge-spectral-Tur\'an type problem, also called the Brualdi--Hoffman--Tur\'an type problem, asks the same question for graphs with a fixed number of edges. 
 Brualdi and Hoffman \cite{BrualdiHoffman} considered the problem of maximizing the spectral radius of a $(0,1)$-matrix with a prescribed number of ones. For symmetric $(0,1)$-matrices with zero diagonal and $2q$ ones, this is the problem of finding the maximum spectral radius among graphs with $q$ edges. Friedland \cite{Friedland1985,Friedland1988} proved several important cases, and Rowlinson \cite{Rowlinson1988} completed the characterization. More precisely, if $q=\binom d2+t$ with $0\le t\le d-1$, apart from isolated vertices, the extremal graph is $H_{d,t}$, where $H_{d,0}=K_d$ and for $t>0$, $H_{d,t}$ is obtained from $K_d$ by adding a single vertex adjacent to exactly $t$ vertices of $K_d$. Spectral Tur\'an problems for cycles and complete bipartite graphs have been studied in \cite{Turan1941,FurediSimonovits2013,Nikiforov2009,Nikiforov2010,Nikiforov2011,ZhaiLinShu2021,SunLi2023,SunLiWei2023}. Further conclusions may be found in these two surveys \cite{LiLiuFeng2022,YuZhangZhang2026}.

The spectral theory of signed graphs has been studied from several perspectives \cite{BelardoEtAl2018,BrunettiStanic2022,GhorbaniMajidi2021,KoledinStanic2017,Stanic2019,Stanic2019Regular,Stanic2020}. Recent spectral Tur\'an-type results concern forbidden unbalanced complete graphs on four vertices, negative triangles, unbalanced complete graphs on $r+1$ vertices with $2\le r\le n-1$
 \cite{ChenYuan2024,WangHouLi2024,XiongHou2024} and negative odd cycles \cite{WangHouHuang2025}. A signed graph $\Gamma$ is called $C_k^-$-free if $\Gamma$ contains no negative cycle of length $k$. Wang and Lin \cite{WangLin2025} obtained fixed-order results for $C_4^-$-free unbalanced signed graphs, while Cai and Zhou \cite{CaiZhou2024} investigated unbalanced signed graphs containing neither $C_3^-$ nor $C_4^-$. Wang, Peng and Hou \cite{WangPengHou2026} determined the unbalanced signed graphs with maximal index for a fixed number of edges.
Motivated by the above developments, we are led to consider the edge-spectral Tur\'an type problem for $C_k^-$-free unbalanced signed graphs.

\begin{problem}\label{p1}
For $m\ge 3$, what are the signed graphs with maximum index among all $C_k^-$-free unbalanced signed graphs with exactly $m$ edges?
\end{problem}

In this paper, we completely resolve the case $k=4$ of Problem~\ref{p1}. For $m\ge3$, let $\mathcal F_m=\{\Gamma=(G,\sigma):|E(G)|=m,\ \Gamma\text{ is unbalanced and }C_4^-\text{-free}\}$. A signed graph $\Gamma\in\mathcal F_m$ is called an extremal signed graph if it attains the maximum index over $\mathcal F_m$, isolated vertices are allowed.
 We first establish the structural characterization of extremal signed graphs in $\mathcal{F}_m$, which is presented in Theorem~\ref{thm1}. 
 
 \begin{theorem}\label{thm1}
 	Up to switching isomorphism and removal of isolated vertices, every extremal signed graph $\Gamma\in\mathcal F_m$ is connected and can be obtained by attaching a negative triangle at some  vertex $u$ of a connected all-positive signed graph $H$ with exactly $m-3$ edges.
 \end{theorem}

 Write uniquely $m-3=\binom d2+t$, where $d\ge1$ and $0\le t\le d-1$. When $t=0$, choose any vertex $u\in V(K_d)$ and denote by $\Gamma_m^*$ the signed graph obtained from $(K_d,+)$ by adding two new vertices $u',u''$ together with the positive edges $uu',uu''$ and the negative edge $u'u''$. When $1\le t\le d-1$, let $z$ be the added vertex of $H_{d,t}$, which is adjacent to precisely $t$ vertices of the underlying complete subgraph $K_d$, with all these incident edges being positive in $(H_{d,t},+)$.  Choose a vertex $u\in N_{H_{d,t}}(z)$. Denote by $\Gamma_m^{**}$ the signed graph obtained from $(H_{d,t},+)$ by adding two vertices $u',u''$, together with the positive edges $uu',uu''$ and the negative edge $u'u''$. Obviously, $\Gamma_m^*\in\mathcal F_m$ for $t=0$, and $\Gamma_m^{**}\in\mathcal F_m$ for $1\le t\le d-1$. These two graphs are shown in Figure~\ref{Fig:1}. In all figures of this paper, dashed lines represent negative edges.
 
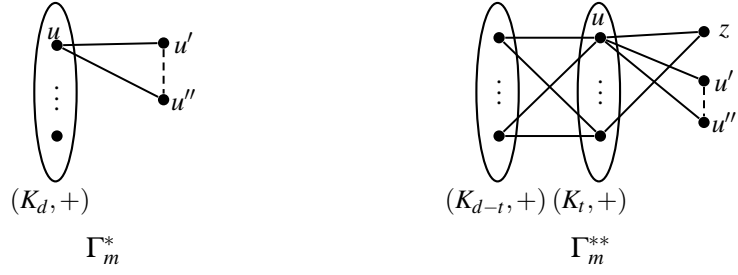
\begin{figure}[H]
	\centering
	\makebox[\textwidth][c]{%
		\begin{minipage}[c]{0.34\textwidth}
			\centering
			\begin{tikzpicture}[
				scale=1,
				vertex/.style={circle, fill=black, inner sep=1.5pt},
				lab/.style={draw=none, fill=none, inner sep=1pt, font=\small},
				positive/.style={line width=0.75pt},
				negative/.style={densely dashed, line width=0.75pt}
				]
				
				\draw[thick] (0.12,0) ellipse (0.28cm and 1.18cm);
				
				\node[vertex] (u) at (0.14,0.62) {};
				\node[vertex] (v) at (0.14,-0.58) {};
				\node[lab] at (0.10,0.78) {$u$};
				\node[lab] at (0.14,0.02) {$\vdots$};
				
				\node[vertex] (x) at (1.55,0.65) {};
				\node[vertex] (y) at (1.55,-0.10) {};
				\node[lab] at (1.82,0.65) {$u'$};
				\node[lab] at (1.82,-0.10) {$u''$};
				
				\draw[positive] (u) -- (x);
				\draw[positive] (u) -- (y);
				\draw[negative] (x) -- (y);
				
				\node[lab] at (0.05,-1.45) {$\left(K_d,+\right)$};
				
			\end{tikzpicture}
			\par\vspace{1pt}
			\(\Gamma_m^*\)
		\end{minipage}%
		
		\hspace{0.015\textwidth}%
		
		\hspace{0.015\textwidth}%
		\begin{minipage}[c]{0.46\textwidth}
			\centering
			\begin{tikzpicture}[
				scale=1,
				vertex/.style={circle, fill=black, inner sep=1.5pt},
				lab/.style={draw=none, fill=none, inner sep=1pt, font=\small},
				positive/.style={line width=0.75pt},
				negative/.style={densely dashed, line width=0.75pt}
				]
				
				\draw[thick] (0.12,0) ellipse (0.28cm and 1.18cm);
				
				\draw[thick] (1.46,0) ellipse (0.28cm and 1.18cm);
				
				\node[vertex] (a1) at (0.14,0.72) {};
				\node[vertex] (a2) at (0.14,-0.58) {};
				\node[lab] at (0.14,0.10) {$\vdots$};
				
				\node[vertex] (u)  at (1.48,0.72) {};
				\node[vertex] (b2) at (1.48,-0.58) {};
				\node[lab] at (1.48,0.10) {$\vdots$};
				\node[lab] at (1.45,0.96) {$u$};
				
				\node[vertex] (z) at (2.85,0.80) {};
				\node[vertex] (x) at (2.85,0.15) {};
				\node[vertex] (y) at (2.85,-0.40) {};
				\node[lab] at (3.12,0.80) {$z$};
				\node[lab] at (3.12,0.15) {$u'$};
				\node[lab] at (3.12,-0.40) {$u''$};
				
				\draw[positive] (a1) -- (u);
				\draw[positive] (a1) -- (b2);
				\draw[positive] (a2) -- (u);
				\draw[positive] (a2) -- (b2);
				
				\draw[positive] (u) -- (z);
				\draw[positive] (u) -- (x);
				\draw[positive] (u) -- (y);
				
				\draw[positive] (b2) -- (z);
				
				\draw[negative] (x) -- (y);
				
				\node[lab] at (0.10,-1.45) {$\left(K_{d-t},+\right)$};
				\node[lab] at (1.34,-1.45) {$\left(K_t,+\right)$};
				
			\end{tikzpicture}
			\par\vspace{1pt}
			\(\Gamma_m^{**}\)
		\end{minipage}%
	}
	
	\par\vspace{1pt}
	\caption{\(\Gamma_m^*\) and \(\Gamma_m^{**}\) 
	}
	\label{Fig:1}
\end{figure}

 With the preparation furnished by Theorem~\ref{thm1}, we go on to derive the full structural characterization of extremal signed graphs, as given in Theorem~\ref{thm2}.

\begin{theorem}\label{thm2}
Let $m\ge3$ with integers $d\ge1$ and $0\le t\le d-1$ giving the unique decomposition $m-3=\binom d2+t$. If $\Gamma\in\mathcal F_m$ attains the maximum index, then the following statements hold:
\begin{enumerate}[label=(\arabic*)]
\item[(i)] If $t=0$, then $\Gamma$ is switching isomorphic to $\Gamma_m^*\cup r_0K_1$ for some integer $r_0\ge0$.
\item[(ii)] If $1\le t\le d-1$, then $\Gamma$ is switching isomorphic to $\Gamma_m^{**}\cup r_0K_1$ for some integer $r_0\ge0$.
\end{enumerate}
\end{theorem}

Inspired by the block matrix formulation of the reconstruction
theorem for signed graphs
\cite{jiang2023maximal},
we use Schur complements in our spectral comparisons.
After the structural reduction in Theorem~\ref{thm1},
we eliminate the eigenvector coordinates corresponding to
the two endpoints of the unique negative edge,
reducing the eigenvalue equation to one for $A(H)$ with
a positive diagonal perturbation at the attachment vertex $u$. The remainder of this paper is organized as follows.
Section~2 presents the preliminary results used in the proofs.
In Section~3, we prove Theorem~\ref{thm1}, showing that every
extremal graph is, up to switching isomorphism and isolated
vertices, obtained by attaching a negative triangle to an
all-positive connected graph $H$.
In Section~4, we use this structural reduction to determine
$H$ and the vertex at which the triangle is attached, thereby
proving Theorem~\ref{thm2} and completing the characterization
of the extremal graphs.
\section{Preliminaries}

\begin{lemma}\cite{Stanic2019}\label{lemma2.1}.
 Let $\Gamma$ be a signed graph. Then there exists a signed graph $\Gamma'$ switching equivalent to $\Gamma$ such that $A(\Gamma')$ has a nonnegative eigenvector corresponding to $\lambda_1(A(\Gamma'))$.
\end{lemma}

\begin{lemma}\cite{XiongHou2024}\label{lemma2.2}.
Let $\Gamma=(G,\sigma)$ be a signed graph with a nonnegative unit
	eigenvector $X=(x_1,x_2,\ldots,x_n)^T$ corresponding to the largest
	eigenvalue $\lambda_1(A(\Gamma))$. If we perform one of the following
	perturbations in $\Gamma$:
	\begin{enumerate}
		\item[(i)] Adding some positive edges,
		\item[(ii)] Removing some negative edges,
		\item[(iii)] Reversing the signs of some negative edges,
	\end{enumerate}
	resulting in a new signed graph $\Gamma'$, then
	$\lambda_1(A(\Gamma'))\geq \lambda_1(A(\Gamma))$.
	The equality holds only if the entries of $X$ corresponding to
	the endpoints of these edges are all zeros.
\end{lemma}
Let $J_d$ be the all-ones matrix, $I_d$ the $d\times d$ identity matrix and $\mathbf1$ the all-ones column vector.
\begin{lemma}\cite{HornJohnson}\label{lemma2.3}.
Let $  M=\begin{pmatrix}P&R\\R^T&S\end{pmatrix}$
be real symmetric.
\begin{enumerate}[label=(\alph*)]
\item[(i)] If $S-\lambda I$ is invertible, then $  M\binom{p}{q}=\lambda\binom{p}{q}$
is equivalent to $
 \left(P-R(S-\lambda I)^{-1}R^T\right)p=\lambda p,
 q=-(S-\lambda I)^{-1}R^Tp.$
\item[(ii)] If $P$ is invertible, then
$ \det(P+uv^T)=\det(P)\bigl(1+v^TP^{-1}u\bigr).$
\item[(iii)]If $P$ is invertible and $1+v^TP^{-1}u\ne0$, then
$ (P+uv^T)^{-1}
 =P^{-1}-\frac{P^{-1}uv^TP^{-1}}{1+v^TP^{-1}u}.$
\item [(iv)]If $P$ is positive definite and $\gamma>0$, 
$ P-\gamma vv^T$ is positive definite if and only if $ \gamma v^TP^{-1}v<1.$
\end{enumerate}
\end{lemma}

\noindent \emph{Proof.}
Parts (i)--(iii) are the standard Schur-complement, matrix-determinant, and Sherman--Morrison formulas. For part (iv), applying a congruence transformation with $P^{-1/2}$ reduces the assertion to the positivity of
$ I-\gamma(P^{-1/2}v)(P^{-1/2}v)^T$,
whose eigenvalues are all equal to $1$, except possibly one equal to $1-\gamma v^TP^{-1}v$.

\begin{lemma}\cite{Stanley1987}\label{lemma2.4}.
Let $\rho_m=\max\{\rho(G): |E(G)|=m\}$,
where $\rho(G)$ denotes the spectral radius of the adjacency matrix of a graph $G$. Then, for every $m\geq 0$, $\rho_m^2+\rho_m\leq 2m.$
\end{lemma}

\begin{lemma}\cite{BrouwerHaemers}\label{lemma2.5}.
Let $T$ be an irreducible nonnegative matrix with spectral radius $\rho(T)$.
Any nonnegative left or right eigenvector of $T$ has eigenvalue $\rho(T)$.
\end{lemma}

	\par\addvspace{1.0\baselineskip}
	\noindent
	\begin{minipage}{\linewidth}
		\centering
		\normalfont
		
		\captionsetup[figure]{
			font=footnotesize,
			labelfont=normalfont,
			textfont=normalfont,
			labelsep=period,
			justification=centering,
			singlelinecheck=false,
			hypcap=false,
			skip=11pt,
			belowskip=0pt
		}
		
		\resizebox{0.92\linewidth}{!}{%
			\begin{tikzpicture}[
				x=.5bp,
				y=-.5bp,
				sg edge/.style={
					draw=black!78,
					line width=.70bp,
					line cap=round
				},
				sg negative/.style={
					sg edge,
					dash pattern=on 3bp off 2.2bp
				},
				sg oval/.style={
					draw=black!42,
					line width=.85bp
				},
				sg label/.style={
					inner sep=0pt,
					outer sep=0pt,
					font=\normalfont\fontsize{11.5bp}{13bp}\selectfont,
					yscale=.80,
					transform shape
				},
				sg name/.style={
					inner sep=0pt,
					outer sep=0pt,
					font=\normalfont\fontsize{15bp}{17bp}\selectfont
				}
				]
				
				\path[use as bounding box]
				(104,306) rectangle (1070,529);
				
				
				\coordinate (a1) at (173.5,369.5);
				\coordinate (a2) at (173.5,428);
				\coordinate (a3) at (252,360);
				\coordinate (a4) at (269.5,388.5);
				\coordinate (a5) at (256.5,425);
				
				\draw[sg oval]
				(173.5,399.5)
				ellipse[x radius=18.5,y radius=46];
				
				\draw[sg edge]
				(a1)--(a3)
				(a1)--(a4)
				(a1)--(a5)
				(a2)--(a3)
				(a2)--(a4);
				
				\draw[sg negative] (a4)--(a5);
				
				\foreach \p in {a1,a2,a3,a4,a5}
				\fill (\p) circle[radius=1.5bp];
				
				\foreach \y in {385.5,397,408.5}
				\fill[black!78]
				(173.5,\y) circle[radius=.8bp];
				
				\node[sg label,anchor=base east]
				at (153,374) {$v_1$};
				
				\node[sg label,anchor=base east,xscale=1.04]
				at (153,431) {$v_{d-2}$};
				
				\node[sg label,anchor=base west]
				at (249,347) {$v_{d-1}$};
				
				\node[sg label,anchor=base west]
				at (279,386) {$v_d$};
				
				\node[sg label,anchor=base west]
				at (257,445) {$v_{d+1}$};
				
				\node[sg label,anchor=base west]
				at (111,467) {$(K_{d-2},+)$};
				
				\node[
				sg name,
				anchor=base west,
				xscale=1.08,
				yscale=.94,
				transform shape
				] at (185,518) {$\dot{G}_{d+1}$};
				
				
				\coordinate (b1) at (492,362.5);
				\coordinate (b2) at (492,421);
				\coordinate (b3) at (575,362.5);
				\coordinate (b4) at (575,421);
				\coordinate (b5) at (643,383);
				
				\draw[sg oval]
				(492,392)
				ellipse[x radius=18.5,y radius=46];
				
				\draw[sg oval]
				(575,392)
				ellipse[x radius=18.5,y radius=46];
				
				\draw[sg edge]
				(b1)--(b3)
				(b1)--(b4)
				(b2)--(b3)
				(b2)--(b4)
				(b3)--(b5)
				(b4)--(b5);
				
				\draw[sg negative]
				(b2)
				.. controls (550,467) and (607,483) ..
				(b5);
				
				\foreach \p in {b1,b2,b3,b4,b5}
				\fill (\p) circle[radius=1.5bp];
				
				\foreach \x in {492,575}
				\foreach \y in {378.5,389.5,400.5}
				\fill[black!78]
				(\x,\y) circle[radius=.8bp];
				
				\node[sg label,anchor=base east]
				at (474,366) {$v_{{t_0}+1}$};
				
				\node[sg label,anchor=base east]
				at (474,428) {$v_d$};
				
				\node[sg label,anchor=base west]
				at (589,357) {$v_1$};
				
				\node[sg label,anchor=base west]
				at (589,434) {$v_{t_0}$};
				
				\node[sg label,anchor=base west]
				at (649,405) {$v_{d+1}$};
				
				\node[sg label,anchor=base west]
				at (393,459) {$(K_{d-{t_0}},+)$};
				
				\node[sg label,anchor=base west]
				at (545,333) {$(K_{t_0},+)$};
				
				\node[
				sg name,
				anchor=base west,
				xscale=1.13,
				yscale=1.05,
				transform shape
				] at (510,514) {$\dot{G}'_e$};
				
				
				\coordinate (c1) at (827.5,362.5);
				\coordinate (c2) at (827.5,421);
				\coordinate (c3) at (911,362.5);
				\coordinate (c4) at (911,421);
				\coordinate (c5) at (989,366);
				\coordinate (c6) at (996,413);
				
				\draw[sg oval]
				(827.5,392)
				ellipse[x radius=18.5,y radius=46];
				
				\draw[sg oval]
				(911,392)
				ellipse[x radius=18.5,y radius=46];
				
				\draw[sg edge]
				(c1)--(c3)
				(c1)--(c4)
				(c2)--(c3)
				(c2)--(c4)
				(c3)--(c5)
				(c3)--(c6)
				(c4)--(c5);
				
				\draw[sg negative] (c5)--(c6);
				
				\foreach \p in {c1,c2,c3,c4,c5,c6}
				\fill (\p) circle[radius=1.5bp];
				
				\foreach \x in {827.5,911}
				\foreach \y in {378.5,389.5,400.5}
				\fill[black!78]
				(\x,\y) circle[radius=.8bp];
				
				\node[sg label,anchor=base east]
				at (809,366) {$v_{t_0}$};
				
				\node[sg label,anchor=base east]
				at (809,428) {$v_d$};
				
				\node[sg label,anchor=base west]
				at (924,357) {$v_1$};
				
				\node[sg label,anchor=base west]
				at (929,434) {$v_{{t_0}-1}$};
				
				\node[sg label,anchor=base west]
				at (1002,366) {$v_{d+1}$};
				
				\node[sg label,anchor=base west]
				at (1014,418) {$v_{d+2}$};
				
				\node[sg label,anchor=base west]
				at (713,460) {$(K_{d-{t_0}+1},+)$};
				
				\node[sg label,anchor=base west]
				at (870,330) {$(K_{{t_0}-1},+)$};
				
				\node[
				sg name,
				anchor=base west,
				xscale=1.23,
				yscale=1.10,
				transform shape
				] at (937,514) {$\dot{G}''_e$};
				
			\end{tikzpicture}%
		}%
		
		\par
		
		
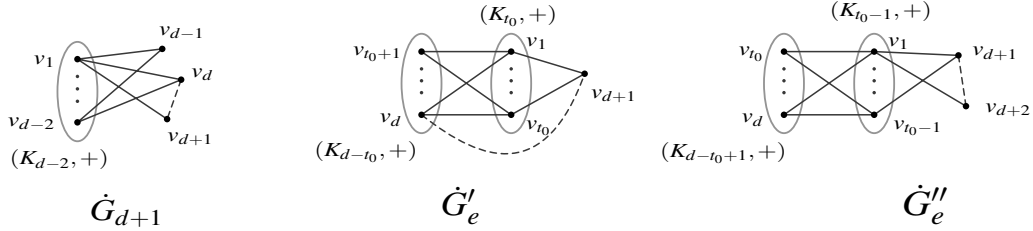
\captionof{figure}{The signed graphs
			$\dot{G}_{d+1}$, $\dot{G}'_e$ and $\dot{G}''_e$.}
		\label{fig2}
		
	\end{minipage}
	\par\addvspace{1.0\baselineskip}

\begin{lemma} \cite{WangPengHou2026}\label{lemma2.6}.
	Let $d,{t_0}$ be positive integers and let $\dot{G}$ be a signed
	graph that maximizes the index in the set of unbalanced signed graphs of size $m$. Up to switching isomorphism,
\end{lemma}

\begin{flushleft}
	(1) if \(m=\binom{d}{2}\), then \(\dot{G}\) has a unique nontrivial component \(\dot{\mathcal{K}^{e}_{d}}\);\\
	(2) if \(m=\binom{d}{2}+1\), then \(\dot{G}\) has a unique nontrivial component \(\dot{G}_{d+1}\);\\
	(3) if \(m=\binom{d}{2}+t_0+1\), where \(1\leq t_0\leq d-2\), let
	\(\theta=\frac{1}{2}\left(3-2d+\sqrt{8d^{2}-8d+1}\right)\), we have\\
	 if \(1\leq t_0\leq \theta-1\), then \(\dot{G}\) has a unique nontrivial component \(\dot{G}^{\prime\prime}_{e}\);\\
	 if \(\theta\leq t_0\leq d-2\), then \(\dot{G}\) has a unique nontrivial component \(\dot{G}^{\prime}_{e}\);\\
	 if \(\theta\) is not integer and \(t_0=\lfloor\theta\rfloor\), then \(\dot{G}\) has a unique nontrivial component is \(\dot{G}^{\prime}_{e}\) or \(\dot{G}^{\prime\prime}_{e}\). See Figure \ref{fig2} for illustrations of $\dot G_{d+1}$, $\dot G_e'$, or $\dot G_e''$. 
\end{flushleft}

Define $E_{vv}=e_ve_v^T$, where $e_v$ is the standard basis vector corresponding to vertex $v$.

\begin{lemma}\label{lemma2.7}
Let $m-3=\binom{d}{2}+t$ with $d\ge1$ and $0\le t\le d-1$. Then
\begin{enumerate}[label=(\roman*)]
\item If $t=0$, then
$\lambda_1(\Gamma_m^*)^2+2\lambda_1(\Gamma_m^*)>2m-4$.
\item If $1\le t\le d-1$, then
$d-1+\frac{t^2(d+1)}{d^3+t^2}<\lambda_1(\Gamma_m^{**})<d$ and
$\lambda_1(\Gamma_m^{**})^2+2\lambda_1(\Gamma_m^{**})>2m-4$.
\end{enumerate}
\end{lemma}
\noindent\emph{Proof.} (i)
	If $t=0$, then $m-3=\binom{d}{2}$, i.e., $2m-6=d(d-1)$. Set $\mu=\lambda_1(\Gamma_m^*)$ for convenience. 
	It is readily seen that $\mu>d-1$, and hence
	$
	\mu^2+2\mu>(d-1)^2+2(d-1)=d^2-1\ge d(d-1)+2=2m-4
	$
	for $d\ge3$. If $d=2$, i.e., $m=4$, then $\Gamma_4^*$ contains a negative triangle with a single positive pendant edge attached at $u$ as a signed subgraph. Thus $\mu^2+2\mu\ge \left(\frac{10}{7}\right)^2 + 2\cdot \frac{10}{7}=240/49>4=2m-4$. If $d=1$, then $m=3$ and $\Gamma_3^*$ is switching equivalent to a negative triangle with exactly a single negative edge with index $\mu=1$. So $\mu^2+2\mu=3>2=2m-4$.
The preceding analysis establishes $\lambda_1(\Gamma_m^{*})^2+2\lambda_1(\Gamma_m^{*})>2m-4$.

 (ii) Label the vertices of $\Gamma_m^{**}$ as $u_1,\ldots,u_d,z,u',u''$, where $u_1=u$ and
$N(z)=\{u_1,\ldots,u_t\}$. For convenience, let $\lambda^*=\lambda_1(\Gamma_m^{**})$. We first prove the lower bound for $\lambda^*$.
Let $\mathbf r$ be a $(d+3)$‑dimensional column vector
such that $r_{u_i}=1$ for $1\le i\le d$ and $r_z=r_{u'}=r_{u''}=0$. Then $\mathbf r^T\mathbf r=d$ and
$\lambda^*\ge\frac{\mathbf r^TA(\Gamma_m^{**})\mathbf r}
{\mathbf r^T\mathbf r}=d-1$.
If $\lambda^*=d-1$, then
$(A(\Gamma_m^{**})\mathbf r)_{u'}=(d-1)r_{u'}=0$.
However,
$(A(\Gamma_m^{**})\mathbf r)_{u'}
=\sigma(u'u)r_u+\sigma(u'u'')r_{u''}=1\cdot1+(-1)\cdot0=1$,
a contradiction. Hence $\lambda^*>d-1$.
Set $h=\lambda^*-(d-1)>0$. We next give a lower bound for $h$.
 For computational convenience, we further define $a_0=\lambda^*+1$ and $m^*=a_0^2-2$, so $h=a_0-d$. Then
\begingroup
\newsavebox{\LemmaSevenMatrixBody}
\newsavebox{\LemmaSevenRowBrace}
\setlength{\arraycolsep}{3pt}
\renewcommand{\arraystretch}{1}

\sbox{\LemmaSevenMatrixBody}{$
	\begin{array}{ccccccc|c|cc}
		0 & 1 & \cdots & 1 & 1 & \cdots & 1 & 1 & 1 & 1 \\
		1 & 0 & \cdots & 1 & 1 & \cdots & 1 & 1 & 0 & 0 \\
		\vdots & \vdots & \ddots & \vdots & \vdots
		& & \vdots & \vdots & \vdots & \vdots \\
		1 & 1 & \cdots & 0 & 1 & \cdots & 1 & 1 & 0 & 0 \\
		1 & 1 & \cdots & 1 & 0 & \cdots & 1 & 0 & 0 & 0 \\
		\vdots & \vdots & & \vdots & \vdots
		& \ddots & \vdots & \vdots & \vdots & \vdots \\
		1 & 1 & \cdots & 1 & 1 & \cdots & 0 & 0 & 0 & 0 \\ \hline
		1 & 1 & \cdots & 1 & 0 & \cdots & 0 & 0 & 0 & 0 \\ \hline
		1 & 0 & \cdots & 0 & 0 & \cdots & 0 & 0 & 0 & -1 \\
		1 & 0 & \cdots & 0 & 0 & \cdots & 0 & 0 & -1 & 0
	\end{array}
	$}

\sbox{\LemmaSevenRowBrace}{$
	t\left\{
	\vphantom{\begin{array}{c}0\\1\\\vdots\\1\end{array}}
	\right.
	$}

\[
A(\Gamma_m^{**})=\quad
\begin{array}{@{}c@{}}
	\begin{array}{@{}c@{}c@{}c@{}}
		\overbrace{\hphantom{\begin{array}{ccccccc}
					0 & 1 & \cdots & 1 & 1 & \cdots & 1
		\end{array}}}^{V(K_d)}
		&
		\overbrace{
			\hphantom{\begin{array}{|c|}1\end{array}}
		}^{\{z\}}
		&
		\overbrace{
			\hphantom{\begin{array}{cc}-1 & -1\end{array}}
		}^{\{u',u''\}}
	\end{array}
	\\[1mm]
	\mathord{\llap{%
			\raisebox{
				\dimexpr\ht\LemmaSevenMatrixBody-\ht\LemmaSevenRowBrace\relax
			}{%
				\usebox{\LemmaSevenRowBrace}%
			}\kern0.5em
	}}%
	\left(\usebox{\LemmaSevenMatrixBody}\right)
\end{array}
=\begin{pmatrix}
	A(K_d)&\mathbf v_t&B\\
	\mathbf v_t^T&0&\mathbf0\\
	B^T&\mathbf0&A_{u'u''}
\end{pmatrix},
\]
\endgroup
where
\[
\begin{gathered}
	A(K_d)=J_d-I_d,\
	\mathbf v_t=(\underbrace{1,\ldots,1}_{t},0,\ldots,0)^T\in\mathbb R^d,\
	B=\begin{pmatrix}
		1&1\\
		0&0\\
		\vdots&\vdots\\
		0&0
	\end{pmatrix}_{d\times2},\
	A_{u'u''}=\begin{pmatrix}0&-1\\-1&0\end{pmatrix}.
\end{gathered}
\]
Let $\mathbf p$ denote a nonzero eigenvector corresponding to
$\lambda^*$, where $\mathbf p_{K_d}$ and $\mathbf p_{u'u''}$ are
its restrictions to $V(K_d)$ and $\{u',u''\}$, respectively, and
$p_z$ is its coordinate at $z$.
By $A(\Gamma_m^{**})\mathbf p=\lambda^*\mathbf p$,
\begin{empheq}[left=\empheqlbrace]{align}
	A(K_d)\mathbf p_{K_d}+\mathbf v_tp_z+B\mathbf p_{u'u''}
	&=\lambda^*\mathbf p_{K_d}, \tag{1}\\
	\mathbf v_t^T\mathbf p_{K_d}
	&=\lambda^*p_z, \tag{2}\\
	B^T\mathbf p_{K_d}+A_{u'u''}\mathbf p_{u'u''}
	&=\lambda^*\mathbf p_{u'u''}. \tag{3}
\end{empheq}
Since $\lambda^*>1$, the matrix $A_{u'u''}-\lambda^*I_2$ is
invertible. From equations (1) and (3), we derive
\[\left[A(K_d)-B(A_{u'u''}-\lambda^*I_2)^{-1}B^T\right]
\mathbf p_{K_d}
=\left[A(K_d)+\frac2{a_0}E_{11}\right]\mathbf p_{K_d}\\
=\lambda^*\mathbf p_{K_d}-\mathbf v_tp_z.\]
 Define
\[
Q_d:=\lambda^*I_d-A(K_d)-\frac2{a_0}E_{11}
=a_0I_d-J_d-\frac2{a_0}E_{11}.
\]
Then $Q_d\mathbf p_{K_d}=\mathbf v_tp_z$.
Since $\lambda^*I-A(\Gamma_m^{**})$ is positive semidefinite and
singular, and $\lambda^*I_2-A_{u'u''}$ is positive definite, its Schur
complement
$
\lambda^*I_{d+1}-A(H_{d,t})-\frac2{a_0}E_{uu}
$ is positive semidefinite and singular.
Note that the matrix $A(H_{d,t})+\frac2{a_0}E_{uu}$ is nonnegative and irreducible.
Let $\mathbf a=(\mathbf p_{K_d}^T,p_z)^T\ne\mathbf0$, and let
$\mathbf b$ denote the vector whose entries are the absolute values of the corresponding entries of $\mathbf a$. Then 
\[\lambda^*\|\mathbf a\|^2
=\mathbf a^T(A(H_{d,t})+\frac2{a_0}E_{uu})\mathbf a
\le\mathbf b^T(A(H_{d,t})+\frac2{a_0}E_{uu})\mathbf b
\le\lambda^*\|\mathbf b\|^2=\lambda^*\|\mathbf a\|^2.\]
Thus 
$(A(H_{d,t})+\frac2{a_0}E_{uu})\mathbf b=\lambda^*\mathbf b$.
Since $\mathbf{b}$ is a nonnegative eigenvector, Lemma~\ref{lemma2.5} yields
$\lambda^*=\rho(A(H_{d,t})+\frac2{a_0}E_{uu})$.
Since $A(K_d)+\frac2{a_0}E_{11}$ is a proper principal submatrix of $A(H_{d,t})+\frac2{a_0}E_{uu}$, the Perron--Frobenius theorem implies 
$
\rho\!\left(A(K_d)+\frac2{a_0}E_{11}\right)$ $<\lambda^*
$.
If $\mu_1\ge\cdots\ge\mu_d$ are the eigenvalues of $A(K_d)+\frac2{a_0}E_{11}$, then the eigenvalues of $Q_d$ are
$\lambda^*-\mu_i>0$. Thus $Q_d$ is positive definite.
Recall $Q_d\mathbf p_{K_d}=\mathbf v_tp_z$.
If $p_z=0$, then $\mathbf p_{K_d}=Q_d^{-1}\mathbf v_tp_z=\mathbf0$, and 
$\mathbf p_{u'u''}=-(A_{u'u''}-\lambda^*I_2)^{-1}B^T\mathbf p_{K_d}=\mathbf0$ by equation (3),
contrary to $\mathbf p\ne\mathbf0$.
Thus $p_z\ne0$.
Substituting  $\mathbf p_{K_d}=Q_d^{-1}\mathbf v_tp_z$ into equation (2) , we obtain  $\lambda^*=\mathbf v_t^TQ_d^{-1}\mathbf v_t$.
Set $P_d=a_0I_d-J_d$.
Since $h>0$, $P_d$ is positive definite with inverse
$
P_d^{-1}=\frac1{a_0}I_d+\frac1{a_0h}J_d$. As $Q_d=P_d-\frac2{a_0}e_1e_1^T$ is positive definite,
 Lemma~\ref{lemma2.3}(iv) implies
$
1-\frac2{a_0}e_1^TP_d^{-1}e_1
=\frac{hm^*-2}{a_0^2h}>0.$ By Lemma~\ref{lemma2.3}(iii),
\[	Q_d^{-1}
=P_d^{-1}
+\frac{\frac2{a_0}P_d^{-1}e_1e_1^TP_d^{-1}}
{1-\frac2{a_0}e_1^TP_d^{-1}e_1}\\
=\frac1{a_0}I_d+\frac1{a_0h}J_d
+\frac{2(he_1+\mathbf1)(he_1+\mathbf1)^T}
{a_0h(hm^*-2)}.\]
Since $\mathbf v_t^T(he_1+\mathbf1)=h+t$,
\[\lambda^*
=\frac{t(h+t)}{a_0h}
+\frac{2(h+t)^2}{a_0h(hm^*-2)}\\
=\frac{h+t}{a_0h}
\left(t+\frac{2(h+t)}{hm^*-2}\right)
=\frac{(h+t)(tm^*+2)}{a_0(hm^*-2)}.\]
Thus $a_0\lambda^*(hm^*-2)=(h+t)(tm^*+2)$, or equivalently,
\[
m^*\bigl(a_0\lambda^*h-t(h+t)\bigr)
=2(a_0\lambda^*+h+t)>0. \tag{4}
\]
Let $h^*=a_0\lambda^*h-t(h+t)$.
Since $m^*>0$, the equation (4) forces
$
h^*>0$ and
$m^*=\frac{2(a_0\lambda^*+h+t)}{h^*}.
$ Recall that $H_{d,t}$ is obtained from $K_d$ by adding one vertex adjacent to exactly $t$ vertices of $K_d$.
Let $\mathbf q=\bigl(\underbrace{1,\ldots,1}_{d},\tfrac{t}{d}\bigr)^T$ be a $(d+1)$-dimensional vector.
Then
$
\mathbf q^T\mathbf q = d+\frac{t^2}{d^2}
$ and
\[\rho(H_{d,t})
\ge\frac{\mathbf q^TA(H_{d,t})\mathbf q}{\mathbf q^T\mathbf q}
=\frac{d(d-1)+2t^2/d}{d+t^2/d^2}\\
=d-1+\frac{t^2(d+1)}{d^3+t^2}.\]
Let $\mathbf h$ be a positive Perron vector of $A(H_{d,t})$. In particular, $h_u>0$.
Construct the extended vector
$\widetilde{\mathbf h}=(\mathbf h^T,0,0)^T$. Then

\[
A(\Gamma_m^{**})\widetilde{\mathbf h}
=\begin{pmatrix}
	\rho(H_{d,t})\mathbf h\\
	h_u\\
	h_u
\end{pmatrix}\text{and }\lambda^*\ge
\frac{\widetilde{\mathbf h}^{\,T}A(\Gamma_m^{**})
	\widetilde{\mathbf h}}
{\widetilde{\mathbf h}^{\,T}\widetilde{\mathbf h}}
=\rho(H_{d,t}).
\]
If $\lambda^*=\rho(H_{d,t})$, then $\widetilde{\mathbf h}$ must be an eigenvector of $A(\Gamma_m^{**})$ corresponding to $\lambda^*$. However, 
$0=\lambda^*\widetilde h_{u'}
=(A(\Gamma_m^{**})\widetilde{\mathbf h})_{u'}=h_u>0$,
a contradiction. Therefore,
\[
\lambda^*>\rho(H_{d,t})
\ge d-1+\frac{t^2(d+1)}{d^3+t^2},\] i.e.,
\[h=\lambda^*-(d-1)>\frac{t^2(d+1)}{d^3+t^2}.
\]

We now prove $\lambda^*<d$.
Since $H_{d,t}$ is a connected graph with $\Delta = d$ and is not
$d$-regular, $\rho(H_{d,t})<d$ and $dI-A(H_{d,t})$ is positive definite. Let
\[
dI-A(H_{d,t})
=\begin{pmatrix}S&-\mathbf v_t\\-\mathbf v_t^T&d\end{pmatrix},\]
where $S=(d+1)I_d-J_d.$ So
\[
\begin{aligned}
	\left[(dI-A(H_{d,t}))^{-1}\right]_{uu}
	&=e_1^TS^{-1}e_1+(e_1^TS^{-1}\mathbf v_t)^2
	(d-\mathbf v_t^TS^{-1}\mathbf v_t)^{-1}\\
	&=\frac2{d+1}
	+\left(\frac{t+1}{d+1}\right)^2
	\frac{d+1}{(d-t)(d+t+1)}\\
	&=\frac2{d+1}
	+\frac{(t+1)^2}{(d+1)(d-t)(d+t+1)}.
\end{aligned}
\]
Since $(d-t)(d+t+1)\ge2d$ and $(t+1)^2\le d^2$,
\[
\begin{aligned}
	\frac2{d+1}\left[(dI-A(H_{d,t}))^{-1}\right]_{uu}
	&=\frac4{(d+1)^2}
	+\frac{2(t+1)^2}{(d+1)^2(d-t)(d+t+1)}\\
	&\le\frac4{(d+1)^2}+\frac d{(d+1)^2}
	=\frac{d+4}{(d+1)^2}<1.
\end{aligned}
\]
By Lemma~\ref{lemma2.3}(iv), the matrix
$dI-A(H_{d,t})-\frac2{d+1}E_{uu}$ is positive definite and $
\rho\!\left(A(H_{d,t})+\frac2{d+1}E_{uu}\right)<d.$ Suppose for contradiction that
$\lambda^*\ge d$. Then
\[
\lambda^*
=\rho\!\left(A(H_{d,t})+\frac2{\lambda^*+1}E_{uu}\right)
\le\rho\!\left(A(H_{d,t})+\frac2{d+1}E_{uu}\right)<d,
\]
a contradiction. Thus $d-1<\lambda^*<d$.

It remains to verify $(\lambda^{*})^2 + 2\lambda^{*} > 2m - 4$.
 Set $\xi=-1+\sqrt{2m-3}$.
Since $m-3\ge0$, we have $\xi>0$ and $\xi^2+2\xi=2m-4$. Define $\mathbf v=\biggl(\,\underbrace{1,\ldots,1}_{V(K_d)},\; \underbrace{\frac{t}{\xi}}_{\{z\}},\; \underbrace{\frac{1}{\xi+1},\frac{1}{\xi+1}}_{\{u',u''\}}\,\biggr)^T$,
and
\[
f(d,t):=\mathbf v^T(A(\Gamma_m^{**})-\xi I)\mathbf v
=d(d-1)-d\xi+\frac{t^2}{\xi}+\frac2{\xi+1}.
\]
We now show that $f(d,t)>0$. First suppose $d\ge5$.
From $(\xi+1)^2=d(d-1)+2t+3$, we have
$(\xi-d+1)^2+2d(\xi-d+1)=2t+3-d$, which yields
\[
d(d-1)-d\xi
=\frac{d-2t-3+(\xi-d+1)^2}{2}
\ge\frac{d-2t-3}{2}.
\]
Also,
\[
(\xi+1)^2\le d(d-1)+2(d-1)+3
=d^2+d+1<(d+1)^2,
\]
so $\xi<d$. Therefore
\[
\begin{aligned}
	f(d,t)
	&=d(d-1)-d\xi+\frac{t^2}{\xi}+\frac2{\xi+1}\\
	&>\frac{d-2t-3}{2}+\frac{t^2}{d}+\frac2{d+1}\\
	&=\frac{(t-d/2)^2}{d}
	+\frac{d^2-5d+2}{4(d+1)}\\
	&>0.
\end{aligned}
\]
For $d=2,3,4$, only finitely many values of $t$ remain.
If $d=2$, the minimum is attained at $t=1$, and
$
\min_{t\in\{1\}}f(2,t)=f(2,1)
=\frac{25}{6}-\frac{65\sqrt7}{42}>0.
$
If $d=3$, the minimum is attained at $t=1$, and
$
\min_{t\in\{1,2\}}f(3,t)=f(3,1)
=\frac{91}{10}-\frac{299\sqrt{11}}{110}>0.
$
If $d=4$, the minimum is attained at $t=2$, and
$
\min_{t\in\{1,2,3\}}f(4,t)=f(4,2)
=\frac{146}{9}-\frac{628\sqrt{19}}{171}>0.
$
Thus $f(d,t)>0$ in every case, and
\[
\lambda^*\ge
\frac{\mathbf v^TA(\Gamma_m^{**})\mathbf v}{\mathbf v^T\mathbf v}
=\xi+\frac{f(d,t)}{\mathbf v^T\mathbf v}>\xi.
\]
Consequently, $(\lambda^*)^2+2\lambda^*>\xi^2+2\xi=2m-4$.

\section{Proof of Theorem \ref{thm1}}
In this section, we first show that, up to switching isomorphism and isolated vertices,
any extremal graph is connected and contains exactly a single negative edge.
We next prove that the two endpoints of this unique negative edge have a unique common neighbor and have no other positive neighbors. Taken together, these facts yield the asserted structure.

\medskip
\noindent \emph{\textbf{Proof of Theorem \ref{thm1}}.}
Let $\Gamma\in\mathcal F_m$ attain the maximum index.
If $m=3$, then every member of $\mathcal F_3$ is a negative triangle apart from isolated vertices. So the statement holds trivially.
Assume $m\ge4$ henceforward.
By Lemma~\ref{lemma2.1}, there exists a signed graph $\widetilde{\Gamma}$ switching equivalent to $\Gamma$ such that an eigenvector corresponding to $\lambda_1(\widetilde{\Gamma})$ is nonnegative. Normalize this eigenvector to have unit norm and denote it by $\mathbf q$ for $\lambda=\lambda_1(\widetilde{\Gamma})$, and let $q_v$ denote the entry of $\mathbf q$ corresponding to $v$.

\begin{claim}\label{claim1}	Apart from isolated vertices, $\widetilde{\Gamma}$ is connected.
\end{claim}
\noindent \emph{Proof.} Otherwise, choose two nontrivial connected components $\Gamma_1,\Gamma_2$ such that the restriction $\mathbf q_1=\mathbf q|_{V(\Gamma_1)}\ne\mathbf0$ and $\lambda_1(\Gamma_1)=\lambda_1(\widetilde{\Gamma})=\lambda$. Choose $u\in V(\Gamma_1)$ such that
$q_u=\max_{v\in V(\Gamma_1)}q_v>0$, and let
$w\in V(\Gamma_2)$.
We construct a new signed graph $\Gamma'$ from $\widetilde{\Gamma}$ by identifying $u$ and $w$ into a single vertex $x'$. The sign function $\sigma'$ on $\Gamma'$ is naturally inherited: every edge not incident to $x'$ retains its original sign, and edges incident to the new vertex $x'$ inherit the signs from the corresponding edges incident to $u$ or $w$ in $\widetilde{\Gamma}$. Clearly, $|E(\Gamma')|=|E(\widetilde{\Gamma})|=m$, and $\Gamma'$ contains no new cycles. So $\Gamma'\in\mathcal F_m$. Extend $\mathbf q_1$ to a new vector
$\widetilde{\mathbf q}_1=(\tilde q_v)_{v\in V(\Gamma')}$, where $\tilde q_v = q_v$ for all $v\in V(\Gamma_1)\setminus \{u\}$, $\tilde q_{x'}= q_u$, and $\tilde q_v = 0$ for all other vertices.
Then
$\widetilde{\mathbf q}_1^{\,T}\widetilde{\mathbf q}_1
=\mathbf q_1^T\mathbf q_1$ and $\lambda_1(\Gamma') \ge \frac{\widetilde{\mathbf q}_1^T A(\Gamma') \widetilde{\mathbf q}_1}{\widetilde{\mathbf q}_1^T \widetilde{\mathbf q}_1} = \frac{\mathbf q_1^T A(\Gamma_1) \mathbf q_1}{\mathbf q_1^T \mathbf q_1} = \lambda.$
If $\lambda_1(\Gamma')=\lambda$, then $\widetilde{\mathbf q}_1$ is an eigenvector of $A(\Gamma')$ corresponding to $\lambda$. However, for any $z\in N_{\Gamma_2}(w)$, we get
\[
\begin{aligned}
	0=\lambda\tilde q_z
	&=(A(\Gamma')\widetilde{\mathbf q}_1)_z
	=\sum_{v\in N_{\Gamma'}(z)}\sigma'(zv)\tilde q_v\\
	&=\sigma'(zx')\tilde q_{x'}
	=\sigma(zw)q_u\ne0,
\end{aligned}
\]
a contradiction. Hence
$\lambda_1(\Gamma')>\lambda=\lambda_1(\widetilde{\Gamma})$,
which contradicts the extremality of $\widetilde{\Gamma}$.
Therefore, apart from isolated vertices,
$\widetilde{\Gamma}$ is connected.

We henceforth discard isolated vertices and retain the same notation. Define $E^+(\Gamma)=\{e\in E(G):\sigma(e)=+1\}$ and $E^-(\Gamma)=\{e\in E(G):\sigma(e)=-1\}$. Let $q_{w}=\max_{v\in V(\widetilde{\Gamma})} q_v>0$.

\begin{claim}\label{claim2}
$|E^-(\widetilde{\Gamma})|=1$.
\end{claim}

\noindent\emph{Proof.}
We first assert that every negative edge belongs to every negative
cycle. Otherwise, suppose that $uv\in E^-(\widetilde{\Gamma})$ does not
belong to a negative signed cycle $(C,\sigma)$.
We construct a new signed graph $\Gamma'$ from $\widetilde{\Gamma}$ by deleting $uv$ and adding
a new positive pendant edge $ww'$ with a new vertex $w'$.
Clearly, $|E(\Gamma')|=|E(\widetilde{\Gamma})|=m$, the cycle $(C,\sigma)$ is still a negative cycle of $\Gamma'$, and the new pendant edge creates
no cycle. Thus $\Gamma'\in\mathcal F_m$.
Extend $\mathbf q$ to a vector $\widetilde{\mathbf q}$ on
$V(\Gamma')$ by setting $\tilde q_{w'}=0$ and keeping all other
coordinates unchanged. Then
$\widetilde{\mathbf q}^{\,T}\widetilde{\mathbf q}
=\mathbf q^T\mathbf q=1$ and
\[
\lambda_1(\Gamma')
\ge\widetilde{\mathbf q}^{\,T}A(\Gamma')\widetilde{\mathbf q}
=\mathbf q^TA(\widetilde{\Gamma})\mathbf q+2q_uq_v
=\lambda+2q_uq_v\ge\lambda.
\]
If $q_u q_v=0$, then $\lambda_1(\Gamma')=\lambda$ by extremality, and the Rayleigh equality condition forces $A(\Gamma')\widetilde{\mathbf q}=\lambda\widetilde{\mathbf q}$. However,
$
0=\lambda \widetilde q_{w'} = \big(A(\Gamma')\widetilde{\mathbf q}\big)_{w'} = q_w>0,
$
a contradiction. If $q_u q_v>0$, then $\lambda_1(\Gamma')>\lambda$, contrary
to the extremality of $\widetilde{\Gamma}$.
Therefore every negative edge belongs to every negative cycle. Let $(C,\sigma)$ be a negative cycle of $\widetilde{\Gamma}$.
Then $E^-(\widetilde{\Gamma})\subseteq E(C)$ and
$(-1)^{|E^-(\widetilde{\Gamma})|}=\sigma(C)=-1$, so $|E^-(\widetilde{\Gamma})|$ is odd.
Suppose for contradiction that $|E^-(\widetilde{\Gamma})|\ge3$. Since $\mathbf q^T\mathbf q=1$, we have
\[
\begin{aligned}
	\lambda
	&=\mathbf q^TA(\widetilde{\Gamma})\mathbf q
	=2\sum_{uv\in E^+(\widetilde{\Gamma})}q_uq_v
	-2\sum_{uv\in E^-(\widetilde{\Gamma})}q_uq_v\\
	&\le2\sum_{uv\in E^+(\widetilde{\Gamma})}q_uq_v
	\le\rho_{m-|E^-(\widetilde{\Gamma})|}\le\rho_{m-3}.
\end{aligned}
\]
 Now let $H$ be an ordinary graph with $m-3$ edges achieving
$\rho(H)=\rho_{m-3}$, and $\mathbf z\ge\mathbf0$ be a unit
Perron vector of $H$. Choose $a\in V(H)$ such that $z_a>0$. We construct a new signed graph  $\Lambda$ from $(H,+)$ by adding two
new vertices $b,c$ together with the positive edges $ab,ac$ and the negative
edge $bc$. Clearly, $|E(\Lambda)|=m$ and the only cycle involving $b$ or $c$ is the negative triangle formed by $a,b,c$, which is the unique negative cycle of $\Lambda$, so $\Lambda\in\mathcal F_m$.
Extend $\mathbf z$ to a vector $\widetilde{\mathbf z}$ on
$V(\Lambda)$ by setting $\tilde z_b=\tilde z_c=0$ and keeping
all other coordinates unchanged. Then
$\widetilde{\mathbf z}^{\,T}\widetilde{\mathbf z}
=\mathbf z^T\mathbf z=1$ and
$
\lambda_1(\Lambda)
\ge\widetilde{\mathbf z}^{\,T}A(\Lambda)\widetilde{\mathbf z}
=\mathbf z^TA(H)\mathbf z=\rho_{m-3}.
$
If $\lambda_1(\Lambda)=\rho_{m-3}$, then
$\widetilde{\mathbf z}$ is an eigenvector of $A(\Lambda)$
corresponding to $\rho_{m-3}$. However,
$
0=\rho_{m-3}\tilde z_b
=(A(\Lambda)\widetilde{\mathbf z})_b
=\tilde z_a-\tilde z_c=z_a>0,
$
a contradiction. Hence
$\lambda_1(\Lambda)>\rho_{m-3}\ge\lambda_1(\widetilde{\Gamma})$,
contrary to the extremality of $\widetilde{\Gamma}$. Thus $|E^-(\widetilde{\Gamma})|=1$.

By Claims~\ref{claim1} and \ref{claim2}, let $u'u''$ be the unique negative edge of $\widetilde{\Gamma}$ with $q_{u'}\ge q_{u''}$. Write $N(v)=N_{\widetilde{\Gamma}}(v)$ and define
$	Z=N(u')\cap N(u''),
	A=N(u')\setminus(Z\cup\{u''\}), B=N(u'')\setminus(Z\cup\{u'\}),
	D=V(\widetilde{\Gamma})\setminus(A\cup B\cup Z\cup\{u',u''\}).$
Since $\widetilde{\Gamma}$ is $C_4^-$-free, we have
$E(A,B)=E(A,Z)=E(B,Z)=E(Z)=\varnothing$.
If $q_{u'}=0$, then $q_{u''}=0$. Since $u'u''$ is the only negative edge, zero coordinates propagate by the eigen-equations. By Claim~\ref{claim1} and $\lambda>0$, we obtain $\mathbf q=\mathbf0$, a contradiction.
Hence $q_{u'}>0$. For every $v\in A\cup Z$, all incident edges are positive, so
$
\lambda q_v=\sum_{z\in N(v)}q_z\ge q_{u'}>0
$ since $u'\in N(v)$, which implies $q_v>0$ for every $v\in A\cup Z$.

For a signed graph $\Gamma$ and a vertex subset $S\subseteq V(\Gamma)$, $\Gamma-S$ is obtained by deleting all vertices in $S$ together with all edges incident to vertices in $S$, and the signs on the remaining edges are unchanged. $\Gamma-e$ is the signed graph on vertex set $V(\Gamma)$ and edge set $E(\Gamma)\setminus\{e\}$, while the signs of all remaining edges are unchanged. 
\begin{claim}\label{claim3}
	$Z\ne\varnothing$.
\end{claim}

\noindent\emph{Proof.}
Assume for contradiction $Z=\varnothing$.
Let $(C,\sigma)$ be a negative cycle of
$\widetilde{\Gamma}$.
Let $u_0\in A$ and $v_0\in B$ be the first and last internal vertices of the positive path $C-u'u''$. We distinguish the following two cases.

\begin{case}
	$\max_{v\in A}q_v<q_{u''}$ and $\max_{v\in B}q_v<q_{u'}$.	\end{case}

Let $a'=|A|$ and $b'=|B|$.
Since $q_v>0$ for $v\in A$, we have $q_{u''}>0$.
Set $\theta'=q_{u'}/q_{u''}\ge1$ because $q_{u'}\ge q_{u''}$.
Since 
$
\lambda q_{u'}=-q_{u''}+\sum_{u\in A}q_u$ and $
\lambda q_{u''}=-q_{u'}+\sum_{v\in B}q_v,
$
\[
\sum_{u\in A}q_u=(1+\lambda\theta')q_{u''}<a'q_{u''},\
\sum_{v\in B}q_v=(\theta'+\lambda)q_{u''}<b'\theta' q_{u''}.
\]
Thus $a'>1+\lambda\theta'$ and $b'>1+\lambda/\theta'$.
For every $u\in A$,
$\lambda q_u=q_{u'}+\sum_{z\in N(u)\setminus\{u'\}}q_z\ge q_{u'}$. Since
 $q_u<q_{u''}$, $1\le\theta'<\lambda$.
Let $F=\widetilde{\Gamma}-\{u',u''\}$.
Then $|E(F)|=m-1-a'-b'$.
Write $\mathbf q_F=\mathbf q|_{V(F)}$ for the restriction
of $\mathbf q$ to $V(F)$. Let
$\mathbf1_U$ be the vector on $V(F)$ whose
coordinate at $v$ is $1$ if $v\in U$ and $0$ otherwise.
By the Cauchy--Schwarz inequality,
\[ \sum_{u\in A}q_u^2
\ge\frac{\left(\sum_{u\in A}q_u\right)^2}{a'}
=\frac{(1+\lambda\theta')^2}{a'}q_{u''}^2, \
\sum_{v\in B}q_v^2
\ge\frac{\left(\sum_{v\in B}q_v\right)^2}{b'}
=\frac{(\theta'+\lambda)^2}{b'}q_{u''}^2.\]
Since $A\cap B=\varnothing$ and $\mathbf q^T\mathbf q=1$,
\[
\|\mathbf q_F\|^2
=1-q_{u'}^2-q_{u''}^2
\ge\sum_{u\in A}q_u^2+\sum_{v\in B}q_v^2\\
\ge q_{u''}^2\left[
\frac{(\lambda\theta'+1)^2}{a'}
+\frac{(\theta'+\lambda)^2}{b'}\right].\tag{5}\]
By the eigenvalue equation
$\lambda\mathbf q_F=A(F)\mathbf q_F+q_{u'}\mathbf1_A+q_{u''}\mathbf1_B$,
we have
\[
\begin{aligned}
	\lambda\|\mathbf q_F\|^2
	&=\mathbf q_F^TA(F)\mathbf q_F
	+q_{u'}\sum_{u\in A}q_u+q_{u''}\sum_{v\in B}q_v\\
	&\le\rho_{m-1-a'-b'}\|\mathbf q_F\|^2
	+q_{u''}^2\bigl[\lambda(\theta'^2+1)+2\theta'\bigr].
\end{aligned}\tag{6}
\]
Since $a'>1+\lambda\theta'$ and $b'>1+\lambda/\theta'$, we have
$a'+b'>2+\lambda(\theta'+1/\theta')\ge2+2\lambda$.
By Lemmas~\ref{lemma2.4} and \ref{lemma2.7},
\[
\begin{aligned}
	(\lambda-\rho_{m-1-a'-b'})(\lambda+\rho_{m-1-a'-b'}+1)
	&=\lambda^2+\lambda-\rho_{m-1-a'-b'}^2-\rho_{m-1-a'-b'}\\
	&\ge\lambda^2+\lambda-2m+2(a'+b')+2\\
	&>2(a'+b')-\lambda-2>0.
\end{aligned}\tag{7}
\]
Since $\lambda+\rho_{m-1-a'-b'}+1>0$ and $\rho_{m-1-a'-b'}\ge0$, $0<\lambda-\rho_{m-1-a'-b'}\le\lambda$.
By inequalities (5) and (6),
we obtain
\[
(\lambda-\rho_{m-1-a'-b'})
\left[\frac{(\lambda\theta'+1)^2}{a'}
+\frac{(\theta'+\lambda)^2}{b'}\right]
\le\lambda(\theta'^2+1)+2\theta'.\tag{8}
\]
Define
$
g(\tau)=\frac{\lambda(\tau^2+1)+2\tau}{(\tau+1)^2}$, then $
g'(\tau)=\frac{2(\lambda-1)(\tau-1)}{(\tau+1)^3}\ge0
\ (\tau\ge1)$. We now show that inequality (8) cannot hold. Otherwise, by the Cauchy--Schwarz inequality and inequality (8),
\[
\lambda-\rho_{m-1-a'-b'}
\le\frac{\lambda(\theta'^2+1)+2\theta'}
{(\lambda\theta'+1)^2/a'+(\theta'+\lambda)^2/b'}
\le\frac{a'+b'}{(\lambda+1)^2}g(\theta').
\]
Since $(a'-1)/\lambda>\theta'\ge1$ and $g$ is strictly increasing on
$[1,\infty)$, we have
$\lambda-\rho_{m-1-a'-b'}<\eta$, where
$
\eta=\frac{a'+b'}{(\lambda+1)^2}g\!\left(\frac{a'-1}{\lambda}\right)
=\frac{(a'+b')\lambda(\lambda^2+a'^2-1)}
{(\lambda+1)^2(\lambda+a'-1)^2}.
$
For $\tau\ge1$, we have $g(\tau)\le\lambda$.
Also $a'+b'\le m-1<(\lambda^2+2\lambda+2)/2$, so
$
0<\eta\le\frac{(a'+b')\lambda}{(\lambda+1)^2}
<\frac\lambda2\left(1+\frac1{(\lambda+1)^2}\right)<\lambda.
$
Hence
\[
0<\lambda-\rho_{m-1-a'-b'}<\eta<\lambda.
\]
Define $f(z)=z(2\lambda+1-z)$ for $z\in[0,\lambda]$.
Since $f'(z)=2\lambda+1-2z\ge1>0$, the function $f$ is strictly
increasing on $[0,\lambda]$. Note
$
\frac{\eta}{a'+b'}=\frac{\lambda(\lambda^2+a'^2-1)}
{(\lambda+1)^2(\lambda+a'-1)^2}
\le\frac{\lambda}{(\lambda+1)^2},
$
and define $g_1(z)=2z-\lambda-2-\frac{\eta}{a'+b'}z
(2\lambda+1-\frac{\eta}{a'+b'}z)$.
For $z\ge0$,
\[
g_1'(z)=2-\frac{\eta}{a'+b'}(2\lambda+1)
+2\left(\frac{\eta}{a'+b'}\right)^2z
\ge2-\frac{\lambda(2\lambda+1)}{(\lambda+1)^2}
=\frac{3\lambda+2}{(\lambda+1)^2}>0.
\]
Since $(a'-1)(b'-1)>\lambda^2$, we have
$a'+b'>a'+1+\lambda^2/(a'-1)$.
Direct expansion yields
\[
g_1\!\left(a'+1+\frac{\lambda^2}{a'-1}\right)
=\frac{P(\lambda,a'-1)}
{(a'-1)^2(\lambda+1)^4(\lambda+a'-1)^4},
\]
where
\begin{align*}
	P(\lambda,a'-1)&=
	\lambda^{10}+3\lambda^9(a'-1)^2+3\lambda^9(a'-1)
	+4\lambda^8(a'-1)^3+16\lambda^8(a'-1)^2+16\lambda^8(a'-1)\\
	&+2\lambda^7(a'-1)^4+17\lambda^7(a'-1)^3+22\lambda^7(a'-1)^2+7\lambda^7(a'-1)\\
	&+4\lambda^6(a'-1)^5+22\lambda^6(a'-1)^4+56\lambda^6(a'-1)^3+46\lambda^6(a'-1)^2+2\lambda^6(a'-1)\\
	&+3\lambda^5(a'-1)^6+17\lambda^5(a'-1)^5+36\lambda^5(a'-1)^4+33\lambda^5(a'-1)^3+11\lambda^5(a'-1)^2\\
	&+16\lambda^4(a'-1)^6+56\lambda^4(a'-1)^5+84\lambda^4(a'-1)^4+50\lambda^4(a'-1)^3+2\lambda^4(a'-1)^2\\
	&+3\lambda^3(a'-1)^7+22\lambda^3(a'-1)^6+33\lambda^3(a'-1)^5+18\lambda^3(a'-1)^4+4\lambda^3(a'-1)^3\\
	&+\lambda^2(a'-1)^8+16\lambda^2(a'-1)^7+46\lambda^2(a'-1)^6+50\lambda^2(a'-1)^5+20\lambda^2(a'-1)^4\\
	&+7\lambda (a'-1)^7+11\lambda (a'-1)^6+4\lambda (a'-1)^5+2(a'-1)^7+2(a'-1)^6.
\end{align*}
Every term on the right-hand side is positive since $a'>1$. Thus
$
g_1(a'+b')>g_1\!\left(a'+1+\frac{\lambda^2}{a'-1}\right)>0.
$
By the definitions of $g_1$ and $f$, we have
$	g_1(a'+b')
=2(a'+b')-\lambda-2-\eta(2\lambda+1-\eta)=2(a'+b')-\lambda-2-f(\eta)>0.
$
Combining this with inequality (7),
we obtain
\[
\begin{aligned}
	f(\eta)
	&<2(a'+b')-\lambda-2\\
	&<(\lambda-\rho_{m-1-a'-b'})(\lambda+\rho_{m-1-a'-b'}+1)\\
	&=f(\lambda-\rho_{m-1-a'-b'}).
\end{aligned}
\]
Since $f$ is strictly increasing on $[0,\lambda]$ and
$0<\lambda-\rho_{m-1-a'-b'}<\eta<\lambda$, it follows that
$\eta<\lambda-\rho_{m-1-a'-b'}$, contrary to
$\lambda-\rho_{m-1-a'-b'}<\eta$.
Therefore inequality (8) cannot hold, and Case 1 is impossible.

\medskip
	
\begin{case}
$\max_{v\in A}q_v\ge q_{u''}$ or $\max_{v\in B}q_v\ge q_{u'}$	\end{case}
Without loss of generality, assume $q_{w_0}=\max_{v\in A}q_v\ge q_{u''}$.
Suppose $B\setminus\{v_0\}\ne\varnothing$.
We construct a new signed graph $\Gamma'$ from $\widetilde{\Gamma}$ by deleting the positive edge $u''v$ and adding the positive edge $w_0v$ for each $v\in B\setminus\{v_0\}$. Clearly, $|E(\Gamma')|=|E(\widetilde{\Gamma})|=m$
, the cycle $(C,\sigma)$ is still a negative cycle of $\Gamma'$, and no new negative $4$-cycle
is created since $N_{\Gamma'}(u'')\setminus\{u'\}=\{v_0\}$ and $N_{\Gamma'}(v_0)\cap A=\varnothing$. Thus $\Gamma'\in\mathcal F_m$, and
\[
\lambda_1(\Gamma')\ge\mathbf q^TA(\Gamma')\mathbf q
=\lambda+2(q_{w_0}-q_{u''})\sum_{v\in B\setminus\{v_0\}}q_v
\ge\lambda.
\]
If $\lambda_1(\Gamma')=\lambda$, then ${\mathbf q}$ is an eigenvector of $A(\Gamma')$ corresponding to $\lambda_1(\Gamma')$.
For every $v\in B\setminus\{v_0\}$,
$
0=\bigl((A(\Gamma')-A(\widetilde{\Gamma}))\mathbf q\bigr)_v
=q_{w_0}-q_{u''}.
$
Thus $q_{u''}=q_{w_0}>0$, and
$\lambda q_v=\sum_{z\in N(v)}q_z\ge q_{u''}>0$, i.e., $q_v>0$.
Consequently,
$
0=\bigl((A(\Gamma')-A(\widetilde{\Gamma}))\mathbf q\bigr)_{w_0}
=\sum_{v\in B\setminus\{v_0\}}q_v>0,
$
a contradiction. Hence  $B\setminus\{v_0\}=\varnothing$, that is, $B=\{v_0\}$. Note that
$q_{v_0}=q_{u'}+\lambda q_{u''}\ge q_{u'}$.
Next suppose that  $A\setminus\{u_0\}\ne\varnothing$.
We construct a new signed graph $\Gamma''$ from $\widetilde{\Gamma}$ by deleting the positive edge $u'u$ and adding the positive edge $v_0u$ for each $u\in A\setminus\{u_0\}$. Clearly,
 $|E(\Gamma'')|=|E(\widetilde{\Gamma})|=m$
, the cycle $(C,\sigma)$ is still a negative cycle of $\Gamma''$ and no new negative $4$-cycle
is created since $N_{\Gamma''}(u')\setminus\{u''\}=\{u_0\}$, $N_{\Gamma''}(u'')\setminus\{u'\}=\{v_0\}$ and $u_0v_0\notin E(\Gamma'')$. Thus $\Gamma''\in\mathcal F_m$ and
\[
\lambda_1(\Gamma'')\ge\mathbf q^TA(\Gamma'')\mathbf q
=\lambda+2(q_{v_0}-q_{u'})\sum_{u\in A\setminus\{u_0\}}q_u
\ge\lambda.
\]
If $\lambda_1(\Gamma'')=\lambda$, then $\mathbf q$ is an eigenvector of $A(\Gamma'')$
corresponding to $\lambda_1(\Gamma'')$.
For every $u\in A\setminus\{u_0\}$,
$
0=\bigl((A(\Gamma'')-A(\widetilde{\Gamma}))\mathbf q\bigr)_u
=q_{v_0}-q_{u'}.
$
Thus $q_{v_0}=q_{u'}>0$
and
$
\lambda q_u=\sum_{z\in N_{\widetilde{\Gamma}}(u)}q_z
\ge q_{u'}>0
$
, i.e., $q_u>0$.
Consequently,
$
0=\bigl((A(\Gamma'')-A(\widetilde{\Gamma}))\mathbf q\bigr)_{v_0}
=\sum_{u\in A\setminus\{u_0\}}q_u>0,
$
a contradiction. Therefore $A\setminus\{u_0\}=\varnothing$, that is,
$A=\{u_0\}$. From the above discussion, we have $V(\widetilde{\Gamma})=\{u',u'',u_0,v_0\}\cup D$, where $u'u''$ is the unique negative edge, $u'u_0$ and $u''v_0$ are two positive edges, and $u_0v_0\notin E(\widetilde{\Gamma})$. Restore $q_{u'}\ge q_{u''}$ by relabeling vertices if needed.
We construct a new signed graph $\Lambda$ from $\widetilde{\Gamma}$ by deleting
the positive edge $u''v_0$ and adding the positive edge $u''u_0$. Clearly, $|E(\Lambda)|=|E(\widetilde{\Gamma})|=m$ and the negative triangle formed by vertices $u',u_0,u''$ is a negative cycle in $\Lambda$. So $\Lambda\in\mathcal F_m$. Note that
$
q_{u_0}=\lambda q_{u'}+q_{u''}$ and $
q_{v_0}=\lambda q_{u''}+q_{u'},$ then $
q_{u_0}-q_{v_0}=(\lambda-1)(q_{u'}-q_{u''})\ge0.
$
We first show that $q_{u_0}\ne q_{v_0}$. Suppose instead $q_{u_0}=q_{v_0}$.
Since $\lambda>1$, 
$q_{u'}=q_{u''}>0$, and thus
\[
\lambda_1(\Lambda)
\ge\mathbf q^TA(\Lambda)\mathbf q
=\lambda+2q_{u''}(q_{u_0}-q_{v_0})=\lambda.
\]
By extremality, $\lambda_1(\Lambda)=\lambda$, so
$A(\Lambda)\mathbf q=\lambda\mathbf q$. However,
$
0=\bigl((A(\Lambda)-A(\widetilde{\Gamma}))\mathbf q\bigr)_{v_0}
=-q_{u''}<0,
$
a contradiction. Thus $q_{u_0}\ne q_{v_0}$, i.e., $q_{u_0}>q_{v_0}$. 
If $q_{u''}=0$, then $A(\Lambda)\mathbf q=\lambda\mathbf q$, whereas $0=((A(\Lambda)-A(\widetilde{\Gamma}))\mathbf q)_{u''}=q_{u_0}-q_{v_0}>0$, a contradiction. Thus $q_{u''}>0$, and
\[
\begin{aligned}
	\lambda_1(\Lambda)-\lambda_1( \widetilde{\Gamma})
	&\ge\mathbf q^T\bigl(A(\Lambda)-A( \widetilde{\Gamma})\bigr)\mathbf q\\
	&=2q_{u''}(q_{u_0}-q_{v_0})>0,
\end{aligned}
\]
  contrary to the extremality of $\widetilde{\Gamma}$.
This means that Case 2 is impossible.

Both cases lead to a contradiction, so $Z\neq\varnothing$.

\begin{claim}\label{claim4}
	$B=\varnothing$.
\end{claim}

\noindent\emph{Proof.}
Suppose for contradiction that $B\ne\varnothing$. We construct a new signed graph $\Gamma'$ from $\widetilde{\Gamma}$ by deleting each positive edge $u''v$ and adding new positive edges $u'v$ for all $v\in B$. Clearly,
we have $|E(\Gamma')|=|E(\widetilde{\Gamma})|=m$,
a negative triangle through a vertex of $Z$ remains in $\Gamma'$ and no negative $4$-cycle
is created. Thus $\Gamma'\in\mathcal F_m$. And
$
\lambda_1(\Gamma')
\ge\mathbf q^TA(\Gamma')\mathbf q
=\lambda+2(q_{u'}-q_{u''})\sum_{v\in B}q_v
\ge\lambda.
$
If $\lambda_1(\Gamma')=\lambda$, then $\mathbf q$ is an eigenvector of
$A(\Gamma')$ corresponding to $\lambda_1(\Gamma')$ and
$
0=\bigl((A(\Gamma')-A(\widetilde{\Gamma}))\mathbf q\bigr)_v
=q_{u'}-q_{u''}.
$
Thus $q_{u''}=q_{u'}>0$.
Since 
$
\lambda q_v=\sum_{z\in N_{\widetilde{\Gamma}}(v)}q_z
\ge q_{u''}>0,
$
 $q_v>0$ for every $v\in B$. However,
$
0=\bigl((A(\Gamma')-A(\widetilde{\Gamma}))\mathbf q\bigr)_{u'}
=\sum_{v\in B}q_v>0,
$
a contradiction. Hence
$\lambda_1(\Gamma')>\lambda=\lambda_1(\widetilde{\Gamma})$,
which contradicts the extremality of $\widetilde{\Gamma}$.
Therefore $B=\varnothing$.

Let $e_S = |E(S)|$ stand for the number of edges within the vertex subset $S$,
and $e_{ST}=|E(S,T)|$ the number of edges between disjoint vertex subsets $S$ and $T$.
\begin{claim}\label{claim5}
$c=|Z|=1$.
\end{claim}

\noindent\emph{Proof.}
By Claims~\ref{claim3} and \ref{claim4}, $Z\ne\varnothing$ and $B=\varnothing$.
Suppose for contradiction that $c\ge2$. We distinguish the following two cases.

\medskip
\noindent\textbf{Case 1.} $A=\varnothing$.

Set $F=\widetilde{\Gamma}-\{u',u''\}$.
Then $|E(F)|=m-2c-1$.
Write $\mathbf q_F=\mathbf q|_{V(F)}$ for the restriction
of $\mathbf q$ to $V(F)$. Note $Z$ is an independent set since $\widetilde{\Gamma}$ is $C_4^-$-free
. Let $M_Z$ be the matrix indexed by $V(F)$ with entries
$
(M_Z)_{vw}=\begin{cases}1,&v,w\in Z,\\0,&\text{otherwise}.\end{cases}
$
Note that
$
\lambda q_{u'}=-q_{u''}+\sum_{z\in Z}q_z$ and $
\lambda q_{u''}=-q_{u'}+\sum_{z\in Z}q_z.
$
Since $\lambda>1$, we have
$
(\lambda-1)(q_{u'}-q_{u''})=0$ and $
q_{u'}=q_{u''}=\frac{\sum_{z\in Z}q_z}{\lambda+1}.
$
For every $v\in V(F)$,
\[
\lambda q_v
=\sum_{w\in V(F)}A(F)_{vw}q_w+
\begin{cases}
 \displaystyle\frac{2}{\lambda+1}\sum_{z\in Z}q_z,&v\in Z,\\[1mm]
 0,&v\notin Z.
\end{cases}
\]
Thus
$
\lambda\mathbf q_F=\left(A(F)+\frac2{\lambda+1}M_Z\right)\mathbf q_F.$
Note that $A(F)+\frac2{\lambda+1}M_Z$ is nonnegative and irreducible, and $\mathbf q_F$ is a nonnegative eigenvector since $q_z>0$ for $z\in Z$.
By Lemma~\ref{lemma2.5},
$
\lambda=\rho\!\left(A(F)+\frac2{\lambda+1}M_Z\right).
$ Set $\alpha=2/(\lambda+1)$. We first show that
$\lambda(\lambda-\alpha c)\le2m-4c-2$.
Set $D=V(F)\setminus Z$. If $D=\varnothing$, then $m-2c-1=0$ and
$\lambda=\alpha c$, so the assertion holds. 
Assume $D\neq \varnothing$. By appropriately ordering the vertices, we have
\[
A(F)+\alpha M_Z=
\bordermatrix{
	& Z & D \cr
	Z & \alpha J_c & \mathcal B \cr
	D & \mathcal B^T & A(\widetilde{\Gamma}[D]) \cr
}.
\]
Here $\mathcal B$ is the inter-block adjacency submatrix between vertex subsets $Z$ and $D$, and $\operatorname{tr}(\mathcal B\mathcal B^T)=\|\mathcal B\|_F^2=|E(Z,D)|$.
Let $(\boldsymbol{\xi}^T,\boldsymbol{\zeta}^T)^T$ be a unit Perron vector of this matrix, and set $\|\boldsymbol{\xi}\|^2=\tau$ and $\|\boldsymbol{\zeta}\|^2=1-\tau$. Then
\[
\begin{aligned}
	\lambda &=
	\begin{pmatrix}
		\boldsymbol{\xi}^T & \boldsymbol{\zeta}^T
	\end{pmatrix}
	\begin{pmatrix}
		\alpha J_c & \mathcal B\\
		\mathcal B^T & A(\widetilde{\Gamma}[D])
	\end{pmatrix}
	\begin{pmatrix}
		\boldsymbol{\xi}\\
		\boldsymbol{\zeta}
	\end{pmatrix}\\
	&=\alpha(\boldsymbol 1^T\boldsymbol\xi)^2+2\boldsymbol\xi^T\mathcal B\boldsymbol\zeta+\boldsymbol\zeta^T A(\widetilde{\Gamma}[D])\boldsymbol\zeta\\
	&\le \alpha c\tau+2\|\mathcal B\|_2\sqrt{\tau(1-\tau)}+\rho(\widetilde{\Gamma}[D])(1-\tau)\\
	&\le \alpha c\tau+2\sqrt{|E(Z,D)|}\sqrt{\tau(1-\tau)}+\rho(\widetilde{\Gamma}[D])(1-\tau)\\
	&=(\sqrt{\tau},\sqrt{1-\tau})
	\begin{pmatrix}
		\alpha c & \sqrt{|E(Z,D)|}\\
		\sqrt{|E(Z,D)|} & \rho(\widetilde{\Gamma}[D])
	\end{pmatrix}
	\begin{pmatrix}
		\sqrt{\tau}\\
		\sqrt{1-\tau}
	\end{pmatrix}.
\end{aligned}
\]
So $\lambda\le\rho(Q)$ , where  $Q=\begin{pmatrix}\alpha c & \sqrt{|E(Z,D)|}\\ \sqrt{|E(Z,D)|} & \rho(\widetilde{\Gamma}[D])\end{pmatrix}$.
By Lemma~\ref{lemma2.4}, $\rho(\widetilde{\Gamma}[D])^2+\rho(\widetilde{\Gamma}[D])\le2(m-2c-1-|E(Z,D)|)$.
If $\lambda\le\alpha c$, then $\lambda(\lambda-\alpha c)\le0\le2m-4c-2$.
If $\alpha c<\lambda\le \rho(\widetilde{\Gamma}[D])$, then
$\lambda(\lambda-\alpha c)\le\lambda^2\le \rho(\widetilde{\Gamma}[D])^2\le2(m-2c-1-|E(Z,D)|)\le2m-4c-2$.
It remains to consider $\lambda>\max\{\alpha c,\rho(\widetilde{\Gamma}[D])\}$.
The characteristic polynomial of $Q$ is
$(x-\alpha c)(x-\rho(\widetilde{\Gamma}[D]))-|E(Z,D)|$. By $\lambda\le\rho(Q)$, we have $(\lambda-\alpha c)(\lambda-\rho(\widetilde{\Gamma}[D]))\le |E(Z,D)|$.
Therefore
\[\begin{aligned}
2m-4c-2
&=2|E(Z,D)|+2(m-2c-1-|E(Z,D)|)\\
&\ge2(\lambda-\alpha c)(\lambda-\rho(\widetilde{\Gamma}[D]))+\rho(\widetilde{\Gamma}[D])^2+\rho(\widetilde{\Gamma}[D])
\end{aligned}\]
and
\[\begin{aligned}
	2m-4c-2-\lambda(\lambda-\alpha c)
	&\ge 2(\lambda-\alpha c)(\lambda-\rho(\widetilde{\Gamma}[D]))+\rho(\widetilde{\Gamma}[D])^2+\rho(\widetilde{\Gamma}[D])-(\lambda-\alpha c)\lambda \\
	&= \rho(\widetilde{\Gamma}[D])^2+\big(1-2(\lambda-\alpha c)\big)\rho(\widetilde{\Gamma}[D])+(\lambda-\alpha c)\lambda.
\end{aligned}\]
If $\lambda-\alpha c\le1/2$, then $\rho(\widetilde{\Gamma}[D])^2+\big(1-2(\lambda-\alpha c)\big)\rho(\widetilde{\Gamma}[D])+(\lambda-\alpha c)\lambda\ge 0$. 
If $\lambda-\alpha c>1/2$, then
\[\begin{aligned}
	\rho(\widetilde{\Gamma}[D])^2+\big(1-2(\lambda-\alpha c)\big)\rho(\widetilde{\Gamma}[D])+(\lambda-\alpha c)\lambda
	&\ge (\lambda-\alpha c)\lambda-\left(\lambda-\alpha c-\frac12\right)^2 \\
	&= (\lambda-\alpha c)(\alpha c+1)-\frac14 \\
	&>0.
\end{aligned}\]
Hence $\lambda(\lambda-\alpha c)\le2m-4c-2$ in every case.
Substituting $\alpha=2/(\lambda+1)$, we obtain
$
\lambda\left(\lambda-\frac{2c}{\lambda+1}\right)\le2m-4c-2.
$
Set $g_2(x)=x+\frac12+\frac1{2x}$ for $x>1$.
We split further into two subcases.

\renewcommand{\thesubcase}{1.\arabic{subcase}}
\setcounter{subcase}{0}

\begin{subcase}
	$c\le g_2(\lambda)=\lambda+\frac12+\frac1{2\lambda}$.
\end{subcase}
Choose $u\in Z$ such that $q_u=\max_{z\in Z}q_z>0$.
We construct a signed graph $\Gamma'$ from $\widetilde{\Gamma}$
by deleting $u'z,u''z$ for every $z\in Z\setminus\{u\}$ and adding
the positive edge $uz$ and a new positive pendant edge at $u$.
Clearly, $|E(\Gamma')|=|E(\widetilde{\Gamma})|=m$, the negative triangle formed by vertices $u',u,u''$ remains a negative cycle in $\Gamma'$ and the negative edge $u'u''$
belongs to no other cycle. Thus $\Gamma'\in\mathcal F_m$. Write $\lambda'=\lambda_1(\Gamma')$, and set
 $\alpha'=2/(\lambda'+1)$.
Let $A(K_{1,c-1})$ be indexed by $V(F)$ with center at the common neighbor $u$ of $u'$ and $u''$,
leaves $Z\setminus\{u\}$, and zero entries elsewhere.
A new pendant vertex together with $u,u'$ induces a positive
path $P_3$, so $\lambda'\ge\sqrt2>1$.
By eliminating $u',u''$ from $\widetilde{\Gamma}$ , we obtain
$\lambda=\rho(A(F)+\alpha M_Z)$.
In $\Gamma'$, if we eliminate $u',u''$ and the $c-1$ new pendant vertices, then
the corresponding principal block of $\lambda'I-A(\Gamma')$
is positive definite, and its Schur complement is positive
semidefinite and singular. Hence
\[
\lambda'=\rho\!\left(A(F)+A(K_{1,c-1})+
\left(\alpha'+\frac{c-1}{\lambda'}\right)E_{uu}\right).
\] Next, we prove that $\lambda'>\lambda$. Define
\[
f(\tau)=\rho\!\left(A(F)+A(K_{1,c-1})+
\left(\frac2{\tau+1}+\frac{c-1}{\tau}\right)E_{uu}\right)-\tau.
\ (\tau>1)
\]
Then $f(\lambda')=0$. We aim to show  $f(\lambda)>0$. Set
$T=\sum_{z\in Z\setminus\{u\}}q_z/q_u$, so $0\le T\le c-1$.
Then
\[\begin{aligned}
	f(\lambda)
	&= \rho\!\left(A(F)+A(K_{1,c-1})+\left(\alpha+\frac{c-1}{\lambda}\right)E_{uu}\right)-\lambda\\
	&\ge \frac{\mathbf q_F^T\!\left(A(F)+A(K_{1,c-1})+\left(\alpha+\frac{c-1}{\lambda}\right)E_{uu}\right)\mathbf q_F}{\mathbf q_F^T\mathbf q_F}
	-\frac{\mathbf q_F^T\bigl(A(F)+\alpha M_Z\bigr)\mathbf q_F}{\mathbf q_F^T\mathbf q_F}\\
	&= \frac{\mathbf q_F^T\!\left[A(K_{1,c-1})+\left(\alpha+\frac{c-1}{\lambda}\right)E_{uu}-\alpha M_Z\right]\mathbf q_F}{\mathbf q_F^T\mathbf q_F}\\
	&= \frac{2q_u\sum_{z\in Z\setminus\{u\}} q_z+\left(\alpha+\frac{c-1}{\lambda}\right)q_u^2-\alpha\left(\sum_{z\in Z} q_z\right)^2}{\mathbf q_F^T\mathbf q_F}\\
	&= \frac{2Tq_u^2+\left(\alpha+\frac{c-1}{\lambda}\right)q_u^2-\alpha q_u^2(1+T)^2}{\mathbf q_F^T\mathbf q_F}\\
	&= \frac{q_u^2\!\left(\frac{c-1}{\lambda}+2(1-\alpha)T-\alpha T^2\right)}{\mathbf q_F^T\mathbf q_F}.
\end{aligned}\]
Define $\Phi(T)=\frac{c-1}{\lambda}+2(1-\alpha)T-\alpha T^2$. The function $\Phi$ is concave, so its minimum on $[0,c-1]$ is attained
at one of the endpoints. Moreover,
$
\Phi(0)=\frac{c-1}{\lambda}>0,
\Phi(c-1)=(c-1)\left(\frac1\lambda+2-\alpha(c+1)\right)$ $\ge0,
$
i.e.,
$2\lambda^2+\lambda+1\ge2c\lambda$. Thus $\Phi(T)\ge0$ yielding $f(\lambda)\ge0$. If $f(\lambda)=0$,
then by Rayleigh equality, $\mathbf q_F$ is an eigenvector of $\left(A(F)+A(K_{1,c-1})+\left(\alpha+\frac{c-1}{\lambda}\right)E_{uu}\right)$ for $\lambda$, $T=c-1$, $c=g_2(\lambda)$, and $q_z=q_u$ for every $z\in Z$.
For each $z\in Z\setminus\{u\}$, 
\[
0=\left(\left[A(K_{1,c-1})+
\left(\alpha+\frac{c-1}{\lambda}\right)E_{uu}-\alpha M_Z\right]
\mathbf q_F\right)_z=q_u(1-\alpha c).
\]
Hence $\alpha c=1$.
Since $c=g_2(\lambda)$ , $\alpha(c+1)=2+1/\lambda$.
Then $\alpha=1+1/\lambda$.
Substituting $\alpha=2/(\lambda+1)$ yields
$\lambda^2+1=0$, a contradiction.
Therefore $f(\lambda)>0$. Since the spectral-radius term is nonincreasing, $f(\tau)$ is strictly decreasing. Since $f(\lambda')=0<f(\lambda)$, we get
$\lambda_1(\Gamma')=\lambda'>\lambda
=\lambda_1(\widetilde{\Gamma})$,
contrary to the extremality of $\widetilde{\Gamma}$.

\begin{subcase}
$c>g_2(\lambda)=\lambda+\frac12+\frac1{2\lambda}$.
\end{subcase}
Recall that $
\lambda\left(\lambda-\frac{2c}{\lambda+1}\right)\le2m-4c-2$. Thus $c\le \frac{(2m-2-\lambda^2)(\lambda+1)}{2(\lambda+2)}$.
By Lemma~\ref{lemma2.7} ,
\[c\le\frac{(2m-2-\lambda^2)(\lambda+1)}{2(\lambda+2)} <\frac{(\lambda+1)^2}{\lambda+2}<g_2(\lambda)<c,
\]
 a contradiction.

Both subcases are impossible. Thus $c=1$ under Case 1.

\medskip
\noindent\textbf{Case 2.} $A\ne\varnothing$.

Suppose that $|A|=k\ge1$.
We first show that
$\max_{v\in A}q_v<q_{u''}$ and $\max_{z\in Z}q_z<q_{u'}$. Otherwise,
suppose that there exists $w_0\in A$ satisfying $q_{w_0}\ge q_{u''}$.
Choose $z_0\in Z$. We construct a new signed graph $\Gamma'$ from
$\widetilde{\Gamma}$ by deleting $u''z$ and adding the positive edge
$w_0z$ for every $z\in Z\setminus\{z_0\}$.
Clearly,
$|E(\Gamma')|=|E(\widetilde{\Gamma})|=m$, the negative triangle formed by vertices $u',z_0,u''$ remains a negative cycle in $\Gamma'$ and no negative $4$-cycle is created since $N_{\Gamma'}(u'')\setminus\{u'\}=\{z_0\}$ and
$N_{\Gamma'}(z_0)\cap(A\cup Z)=\varnothing$. Thus $\Gamma'\in\mathcal F_m$, and
\[
 \lambda_1(\Gamma')\ge\mathbf q^TA(\Gamma')\mathbf q
 =\lambda+2(q_{w_0}-q_{u''})\sum_{z\in Z\setminus\{z_0\}}q_z\ge\lambda.
\]
If $\lambda_1(\Gamma')=\lambda$, then 
$A(\Gamma')\mathbf q=\lambda\mathbf q$.
However,
$
 0=\bigl((A(\Gamma')-A(\widetilde{\Gamma}))\mathbf q\bigr)_{w_0}
 =\sum_{z\in Z\setminus\{z_0\}}q_z>0,
$
a contradiction. Hence $\max_{v\in A}q_v<q_{u''}$. Next assume some $z\in Z$ satisfies $q_z\ge q_{u'}$.
We construct a new signed graph $\Gamma''$ from $\widetilde{\Gamma}$ by deleting
$u'v$ and adding the positive edge $zv$ for every $v\in A$. Clearly,
$|E(\Gamma'')|=|E(\widetilde{\Gamma})|=m$,
the negative triangle formed by vertices $u',z,u''$ remains a negative cycle in $\Gamma''$, and no negative $4$-cycle is created.
Thus $\Gamma''\in\mathcal F_m$ and
\[
 \lambda_1(\Gamma'')\ge\mathbf q^TA(\Gamma'')\mathbf q
 =\lambda+2(q_z-q_{u'})\sum_{v\in A}q_v\ge\lambda.
\]
If $\lambda_1(\Gamma'')=\lambda$, then
$A(\Gamma'')\mathbf q=\lambda\mathbf q$.
However,
$
 0=\bigl((A(\Gamma'')-A(\widetilde{\Gamma}))\mathbf q\bigr)_z
 =\sum_{v\in A}q_v>0,
$
a contradiction. Hence $\max_{z\in Z}q_z<q_{u'}$.

 Note that
$
\lambda q_v
=q_{u'}+\sum_{z\in N_{\widetilde{\Gamma}}(v)\setminus\{u'\}}q_z
\ge q_{u'}>0.
$
Thus $q_v>0$ for every $v\in A$
and $
q_{u''}>\max_{v\in A}q_v\ge{q_v}\ge\frac{q_{u'}}{\lambda}>0.$
 Set $q_{u'}=\theta' q_{u''}$. 
 The eigen-equations at $u'$ and $u''$ give
 \begin{align}
 	\displaystyle\sum_{z\in Z} q_z &= (\theta'+\lambda)q_{u''} \tag{9},\\
 	\displaystyle\sum_{v\in A} q_v &= (\lambda-1)(\theta'-1) q_{u''} \tag{10}.
 \end{align}
 So $\theta'>1$. For any $v\in A$, $\lambda q_v\ge q_{u'}=\theta' q_{u''}$ and $q_v<q_{u''}$,
which give $\theta'<\lambda$.
Moreover, $m\ge1+2c+k\ge6$.
The graph $\Gamma_m^*$ for $t=0$ or $\Gamma_m^{**}$ for $t>0$ with $m$ edges contains $\Gamma_6^*$ as an induced
signed subgraph, so
$\lambda\ge\lambda_1(\Gamma_6^*)=\sqrt5>2$.
For every $v\in A$,
$
 \lambda q_v=q_{u'}+\sum_{w\in N(v)\setminus\{u'\}}q_w\ge\theta' q_{u''}.
$ Then
$
\sum_{v\in A}\lambda q_v \ge \sum_{v\in A}\theta' q_{u''}= k\theta' q_{u''}.$
By equation (10), then
$
(\lambda-1)(\theta'-1) \ge \frac{k\theta'}{\lambda}.
$
And $
\sum_{v\in A} q_v < \sum_{v\in A} q_{u''} = k q_{u''}
$ since $\max_{v\in A} q_v < q_{u''}$,
so
$ (\lambda-1)(\theta'-1) < k.
$
Hence
\[
\frac{k\theta'}{\lambda} \le (\lambda-1)(\theta'-1) < k.\tag*{(11)}
\] Similarly, for $z\in Z$,
$
 \lambda q_z=q_{u'}+q_{u''}+\sum_{w\in N(z)\setminus\{u',u''\}}q_w
 \ge(\theta'+1)q_{u''}.
$
By $q_z<q_{u'}=\theta' q_{u''}$ and equation (9), we have
\[
 \frac{c(\theta'+1)}{\lambda}\le\theta'+\lambda<c\theta'.\tag*{(12)}
\]
If $c=2$, then $Z=\{z_1,z_2\}$ and
$
 (\lambda+\theta')q_{u''}=q_{z_1}+q_{z_2}<2\theta' q_{u''}
$,
and so $\lambda<\theta'$, a contradiction.
Now, suppose $c\ge3$. Let $q_w=\max_vq_v$, then either $w\in D$ or $w=u'$ since $\max_{v\in A}q_v<q_{u''}<q_{u'}$ and $\max_{z\in Z}q_z<q_{u'}$.
We treat these two subcases separately.
\renewcommand{\thesubcasee}{2.\arabic{subcasee}}
\renewcommand{\theHsubcasee}{2.\arabic{subcasee}}
\setcounter{subcasee}{0}
\begin{subcasee}
	  $w\in D$.\end{subcasee}
Note 
$
 \lambda q_w=\sum_{v\in N(w)}q_v,$  $\lambda^2 q_w=\sum_{v\in N(w)}\lambda q_v \le \sum_{v\in N(w)}  d(v) q_w$ , so $\lambda^2\le\sum_{v\in N(w)} d(v)$.
Let $n=|V(\widetilde{\Gamma})|$. Thus
\[
\begin{aligned}
 \sum_{v\notin N(w)} d(v)
 &\ge d(w)+ d(u')+ d(u'')+n- d(w)-3\\
 &=n+ d(u')+ d(u'')-3\\
 &=n+2c+k-1.
\end{aligned}
\]
Since $\sum_v d(v)=2m$ and $n\ge c+k+3$, we obtain
$
 \lambda^2\le2m-n-2c-k+1\le2m-2-3c-2k.
$
By Lemma~\ref{lemma2.7},
 $3c+2k<2\lambda+2$.
By equations (11) and (12), we have
 $k>(\lambda-1)(\theta'-1)$ and $c>1+\lambda/\theta'$, so
$
 3c+2k>3+\frac{3\lambda}{\theta'}+2(\lambda-1)(\theta'-1),
$
 that is, $3+\frac{3\lambda}{\theta'}+2(\lambda-1)(\theta'-1)<3c+2k<2\lambda+2$. However,
\[
\begin{aligned}
 3+\frac{3\lambda}{\theta'}+2(\lambda-1)(\theta'-1)-2\lambda-2
 &=3+\frac{3\lambda}{\theta'}+2(\lambda-1)\theta'-4\lambda\\
 &\ge3+2\sqrt{6\lambda(\lambda-1)}-4\lambda>0,\ (\text{since}\ \lambda >2)
\end{aligned}
\]
a contradiction. This excludes the possibility $w\in D$.
\begin{subcasee}
	 $w=u'$
\end{subcasee}
Note that $E(A,Z)=E(Z)=\varnothing$ and
$
 m=1+2c+k+e_A+e_{AD}+e_{ZD}+e_D,
$
then
\[
 \lambda^2q_{u'}=-\lambda q_{u''}+
 \sum_{v\in A\cup Z}\lambda q_v
 \le-\lambda q_{u''}+q_{u'}\sum_{v\in A\cup Z} d(v).
\]
Also
$
 \sum_{v\in A\cup Z} d(v)
 =2c+k+2e_A+e_{AD}+e_{ZD}=m-1+e_A-e_D.
$
Since $q_{u'}=\theta' q_{u''}>0$,
$
 \lambda^2\le m-1+e_A-e_D-\frac\lambda\theta'$.
For $z\in Z$, 
$\lambda q_z=q_{u'}+q_{u''}+\sum_{v\in N(z)\cap D}q_{v}$ 
and
$
 \sum_{zv\in E(Z,D)}q_{v}
 =q_{u''}\bigl[\lambda(\theta'+\lambda)-c(\theta'+1)\bigr].
$
Since $q_{v}\le q_{u'}=\theta' q_{u''}$ for every ${v}\in D$,
$
 e_{ZD}\ge\frac{\lambda(\theta'+\lambda)-c(\theta'+1)}{\theta'}.
$
Moreover,
$
 e_A-e_D=m-1-2c-k-e_{AD}-e_{ZD}-2e_D
 \le m-1-2c-k-e_{ZD}.
$
Consequently,
\[
\begin{aligned}
 \lambda^2
 &\le2m-2-2c-k-e_{ZD}-\frac\lambda\theta'\\
 &\le2m-2-c-k-\lambda+\frac{c-\lambda^2-\lambda}{\theta'}.
\end{aligned}
\]
Since $2m-6<\lambda^2+2\lambda-2$, it follows that
$
 c+k-\lambda-2+\frac{\lambda^2+\lambda-c}{\theta'}<0.
$
However, 
$
 c\le\frac{\lambda(\theta'+\lambda)}{\theta'+1}
 <\frac{\lambda(\lambda+1)}2
$ since $c(\theta'+1)/\lambda\le\theta'+\lambda$,
where $1<\theta'<\lambda$. Thus $\lambda^2+\lambda-c>0$ and
$
 \frac{\lambda^2+\lambda-c}{\theta'}>\lambda+1-\frac c\lambda.
$
Hence
$
 c+k-\lambda-2+\frac{\lambda^2+\lambda-c}{\theta'}
 >c\left(1-\frac1\lambda\right)+k-1>0,
$
a contradiction.

Both cases are impossible. Therefore $c=|Z|=1$. Write $Z=\{u\}$.
\renewcommand{\thecasee}{\arabic{casee}}
\renewcommand{\theHcasee}{\theHclaim.\arabic{casee}}
\setcounter{casee}{0}

\begin{claim}\label{claim6}
$A=\varnothing$.
\end{claim}

\noindent\emph{Proof.}
Otherwise,
suppose that $A\ne\varnothing$. Note $q_u=q_{u'}+\lambda q_{u''}\ge q_{u'}$ since $\lambda q_{u''}=-q_{u'}+q_u$.
We construct a new signed graph $\Gamma'$ from $\widetilde{\Gamma}$ by deleting each positive edge $u'v$ and adding new positive edges $uv$ for all $v\in A$.
Clearly,
$|E(\Gamma')|=|E(\widetilde{\Gamma})|=m$, the negative triangle formed by vertices $u',u,u''$ remains a negative cycle in $\Gamma'$ and no negative
$4$-cycle is created since $N_{\Gamma'}(u')\setminus\{u''\}=N_{\Gamma'}(u'')\setminus\{u'\}=\{u\}$. Thus $\Gamma'\in\mathcal F_m$.
Since $q_u\ge q_{u'}$, we have
$
\lambda_1(\Gamma')
\ge\mathbf q^TA(\Gamma')\mathbf q
=\lambda+2(q_u-q_{u'})\sum_{v\in A}q_v
\ge\lambda.
$
If $\lambda_1(\Gamma')=\lambda$, then $\mathbf q$ is an eigenvector of
$A(\Gamma')$ corresponding to $\lambda$.
However,
$
0=\bigl((A(\Gamma')-A(\widetilde{\Gamma}))\mathbf q\bigr)_u
=\sum_{v\in A}q_v>0,
$
a contradiction. Hence
$\lambda_1(\Gamma')>\lambda=\lambda_1(\widetilde{\Gamma})$,
which contradicts the extremality of $\widetilde{\Gamma}$.
Therefore $A=\varnothing$.

By Claims~\ref{claim4}--\ref{claim6}, $A=B=\varnothing$ and $Z=\{u\}$. Put $H=\widetilde{\Gamma}-\{u',u''\}$. All edges of $H$ are positive and $|E(H)|=m-3$. 
By Claim~\ref{claim1}, every connected component of $H$ must contain $u$ since $u',u''$ have no positive neighbor other than $u$. Hence $H$ is connected, and $\widetilde{\Gamma}$ is obtained from $(H,+)$ by adding the positive edges $uu',uu''$ and the negative edge $u'u''$. This completes the proof of Theorem~\ref{thm1}.

\section{Proof of Theorem \ref{thm2}}
Theorem~\ref{thm1} proves that any extremal graph is constructed by attaching a negative triangle to an all-positive signed graph $H$ with $m-3$ edges. In the subsequent proof, we first consider the case $t=0$, and then discuss the case $1\le t\le d-1$ by characterizing the properties of $H$. 

Consider a partition $\mathcal{P}=\{V_1,V_2,\dots,V_l\}$ of
the vertex set $V$. Let $A=(a_{uv})_{u,v\in V}$ be a real
symmetric matrix whose rows and columns are partitioned
according to $\mathcal{P}$. Write $A_{ij}$ for the submatrix
formed by the rows indexed by $V_i$ and columns indexed by
$V_j$, and let $q_{ij}$ be the average row sum of $A_{ij}$,
namely,
$q_{ij}=\frac{1}{|V_i|}\sum_{u\in V_i}\sum_{v\in V_j}a_{uv}$.
The $l\times l$ matrix $Q=(q_{ij})_{1\le i,j\le l}$ is called
the quotient matrix of $A$ with respect to the partition
$\mathcal{P}$. The partition is equitable if each submatrix
$A_{ij}$ has constant row sums.

\setcounter{lemma}{7}
\begin{remark}\cite{WangPengHou2026}\label{remark8}.
	Let $m=\binom d2+{t_0}+1$, where $d\ge4$ and $2\le {t_0}\le d-2$.
	For the signed graphs $\dot G'_e$ and $\dot G''_e$ in
	Lemma~\ref{lemma2.6}, let $Q_{\dot G'_e}$ and $Q_{\dot G''_e}$ be the quotient matrices of $\dot G'_e$ and $\dot G''_e$, respectively.
	Then $\lambda_1(\dot G'_e)$ and $\lambda_1(\dot G''_e)$ are the
	largest zeros of $f(Q_{\dot G'_e},x)$ and $f(Q_{\dot G''_e},x)$,
	respectively, where
	\[
	\begin{aligned}
		f(Q_{\dot G'_e},x)
		&=x^4+(3-d)x^3+(2-2d-{t_0})x^2\\
		&+(-2+d{t_0}-{t_0}^2)x-{t_0}^2+d{t_0}+d+{t_0}-2,\\
		f(Q_{\dot G''_e},x)
		&=x^5+(3-d)x^4+(2-2d-{t_0})x^3+(-2+d{t_0}-{t_0}^2)x^2\\
		&+(d{t_0}-{t_0}^2+4{t_0}-5)x+{t_0}^2-d{t_0}+d-1.
	\end{aligned}
	\]
	
\end{remark}
\medskip
\noindent \emph{\textbf{Proof of Theorem \ref{thm2}}.} We first prove part (i), where $t=0$ and $m-3=\binom d2$.
If $d=1$, then $m=3$ and $\widetilde{\Gamma}$ is switching isomorphic to
$\Gamma_3^*$, apart from isolated vertices.
If $d=2$, then $m=4=\binom32+1$.
By Lemma~\ref{lemma2.6}(2),  $\widetilde{\Gamma}$ is switching isomorphic to $\dot G_4=\Gamma_4^*$,
apart from isolated vertices.
If $d\ge4$, then $m=\binom d2+3$.
By Lemma~\ref{lemma2.6}(3), with ${t_0}=2$,
every unrestricted maximizer is switching isomorphic either to 
$\dot G'_e$ or to $\dot G''_e=\Gamma_m^*$, apart from isolated vertices.
By Remark~\ref{remark8}, with ${t_0}=2$, the indices of
$\dot G'_e$ and $\dot G''_e$ are the largest zeros of
$f_1(x)$ and $f_2(x)$, respectively, where
$
f_1(x)=x^4+(3-d)x^3-2dx^2+(2d-6)x+3d-4
$ and
$
f_2(x)=(x^2-1)
\bigl[x^3+(3-d)x^2+(1-2d)x+d-3\bigr].
$
We obtain $xf_1(x)-f_2(x)=(d-3)(x+1)$. Then $f_2\big(\lambda_1(\dot G_e')\big) = -(d-3)\big(\lambda_1(\dot G_e')+1\big) < 0$, so $\lambda_1(\Gamma_m^*) > \lambda_1(\dot G_e').$
 Thus $\widetilde{\Gamma}$ is switching isomorphic to $\Gamma_m^*$, apart from isolated vertices. It remains to consider $d=3$ and then $m=6$.
By Theorem~\ref{thm1}, $H$ is a connected graph with three edges,
and hence $H\in \{ K_3, P_4, K_{1,3}\}.$
By a direct calculation of the characteristic polynomials of
the resulting signed graphs, we have
$\lambda_1(\widetilde{\Gamma})\le\sqrt5
=\lambda_1(\Gamma_6^*)$,
with equality if and only if $H=K_3$.
 Hence $\widetilde{\Gamma}$ is switching
isomorphic to $\Gamma_6^*$, apart from isolated vertices. Combining all cases, $\widetilde{\Gamma} $ is switching isomorphic to $\Gamma_m^* \cup r_0K_1$ for some integer $r_0\ge 0$ if $t=0$.

We next prove part (ii) for $1\le t\le d-1$, where $m-3=\binom d2+t$. The vertices $u',u$ and any neighbor of $u$ in $H$ induce a positive path $P_3$, so $\lambda\ge\sqrt2>1$. Recall that  $\lambda^*=\lambda_1(\Gamma_m^{**})$.

\begin{claim}\label{claim7}
$N_H(u)=V(H)\setminus\{u\}$.
\end{claim}
\noindent\emph{Proof.}
	 We first show $q_u=\max_{v\in V(H)} q_v$. Otherwise,
	suppose that there exists a vertex $v\in V(H)$ such that $q_v>q_u$. We construct a new signed graph $\Gamma'$ from $\widetilde{\Gamma}$ by deleting the positive edges $uu',uu''$ and adding the positive edges $vu',vu''$. Clearly,
$|E(\Gamma')|=|E(\widetilde{\Gamma})|=m$,
the negative triangle formed by vertices $u',v,u''$ is the unique negative cycle in $\Gamma'$ and no negative
$4$-cycle is created.
Then $\Gamma'\in\mathcal F_m$, and
	\[
	\lambda_1(\Gamma')\ge \mathbf q^T A(\Gamma')\mathbf q=\lambda+2(q_v-q_u)(q_{u'}+q_{u''})>\lambda,
	\]
	a contradiction. Hence $q_u=\max_{v\in V(H)} q_v$. Now suppose that there exists a vertex $v\in V(H)\setminus(N_H(u)\cup\{u\})$.
 Let $u=v_0v_1\cdots v_\ell=v$ be a shortest path in $H$ with $\ell\ge2$.
	We construct a new signed graph $\Gamma'$ from $\widetilde{\Gamma}$ by deleting the positive edge $v_{\ell-1}v_\ell$ and adding the positive edge $uv_\ell$.
Clearly,
$|E(\Gamma')|=|E(\widetilde{\Gamma})|=m$, the negative triangle formed by vertices $u',u,u''$ remains a negative cycle in $\Gamma'$ and no negative
$4$-cycle is created. Then $\Gamma'\in\mathcal F_m$, and
	\[
	\lambda_1(\Gamma')\ge \mathbf q^T A(\Gamma')\mathbf q=\lambda+2q_{v_\ell}(q_u-q_{v_{\ell-1}})\ge\lambda.
	\]
	If $\lambda_1(\Gamma')=\lambda$, then
	$A(\Gamma')\mathbf q=\lambda\mathbf q$.
    We have
	$
	0=\big((A(\Gamma')-A(\widetilde{\Gamma}))\mathbf q\big)_{v_\ell}
	=q_u-q_{v_{\ell-1}}.
	$
	Thus
	$
	\lambda q_{v_\ell}
	=\sum_{z\in N_H(v_\ell)}q_z
	\ge q_{v_{\ell-1}}=q_u>0.
	$
	However,
	$
	0=\big((A(\Gamma')-A(\widetilde{\Gamma}))\mathbf q\big)_u
	=q_{v_\ell}>0
	$,
	a contradiction. Thus $N_H(u)=V(H)\setminus\{u\}$.
		
By Claim~\ref{claim7}, for every $v\in V(H)\setminus\{u\}$,
we have
$\lambda q_v=\sum_{z\in N_H(v)}q_z\ge q_u>0$.
Thus $q_v>0$ for every $v\in V(H)$.
Order vertices of $V(H)=\{v_1,\dots,v_{n'}\}$ with $v_1=u$ such that $q_{v_1}\ge q_{v_2}\ge\cdots\ge q_{v_{n'}}>0$. 
	
\begin{claim}\label{claim8}
 Let $v_i, v_j, v_k \in V(H)$ be three distinct vertices with $i<j$. If $v_jv_k\in E(H)$, then $v_iv_k\in E(H)$.	\end{claim}
\noindent\emph{Proof.}	
If $i=1$ or $k=1$, then the conclusion follows directly from Claim~\ref{claim7}. We assume $2\le i<j$ and $k\ne 1$. Suppose for contradiction that $v_iv_k\notin E(H)$.
We construct a new signed graph $\Gamma'$ from $\widetilde{\Gamma}$ by deleting the positive edge $v_jv_k$ and adding the positive edge $v_iv_k$.
Clearly,
$|E(\Gamma')|=|E(\widetilde{\Gamma})|=m$,
the negative triangle formed by vertices $u',u,u''$ remains a negative cycle in $\Gamma'$ and no negative
$4$-cycle is created. 	Then $\Gamma'\in\mathcal F_m$, and
\[
\lambda_1(\Gamma')\ge \mathbf q^T A(\Gamma')\mathbf q=\lambda+2q_{v_k}(q_{v_i}-q_{v_j})\ge\lambda.
\]
If $\lambda_1(\Gamma')=\lambda$, then $(A(\Gamma')-A(\widetilde{\Gamma}))\mathbf q=\mathbf 0$.
However,
$
0=\big((A(\Gamma')-A(\widetilde{\Gamma}))\mathbf q\big)_{v_i}=q_{v_k}>0,
$
a contradiction. Therefore $v_iv_k\in E(H)$.

Let $c'$ be the order of a maximum clique in $H$. By Claim \ref{claim8}, $\{v_1,\ldots,v_{c'}\}$ induces a clique.
Thus $2\le c'\le d$. Otherwise, if $c'\ge d+1$, then
$
m\ge\binom{c'}2+3\ge\binom{d+1}2+3
=\binom d2+d+3>\binom d2+t+3=m
$,
a contradiction. By the eigen-equations at $u'$ and $u''$, we have
$\lambda q_{u'}=q_u-q_{u''}$ and
$\lambda q_{u''}=q_u-q_{u'}$.
Since $\lambda>1$, we obtain
$q_{u'}=q_{u''}=q_u/(\lambda+1)$.
Let $\mathbf q_H=\mathbf q|_{V(H)}$. Then
$
\lambda\mathbf q_H
=A(H)\mathbf q_H+(q_{u'}+q_{u''})e_u
=\left(A(H)+\frac2{\lambda+1}E_{uu}\right)\mathbf q_H.
$
Since $H$ is connected, the matrix
$A(H)+\frac2{\lambda+1}E_{uu}$ is nonnegative and irreducible.
Moreover, $\mathbf q_H$ is positive. By Lemma~\ref{lemma2.5},
$
\lambda_1(\widetilde{\Gamma})
=\lambda
=\rho\!\left(A(H)+\frac2{\lambda+1}E_{uu}\right).
$
 Recall $a_0=\lambda^*+1$, and
 $
 \rho\!\left(A(H)+\frac2{a_0}E_{uu}\right)
 \ge
 \rho\!\left(A(H)+\frac2{\lambda+1}E_{uu}\right)
 =\lambda\ge\lambda^*
 $	since $a_0\le\lambda+1$.

\begin{claim}\label{claim9}
	$H\cong H_{d,t}$.
\end{claim}
\noindent\emph{Proof.}
Let $W=V(H)\setminus\{v_1,\ldots,v_{c'}\}=\{w_1,\ldots,w_s\}$,
where $w_i=v_{c'+i}$ for $1\le i\le s$ and $s=n'-c'$.
We first show that $W$ is an independent set. Otherwise, suppose that
$v_iv_j\in E(H)$ with $c'<i<j$. For every $1\le a<b\le c'+1$,
Claim~\ref{claim8} implies $v_av_j\in E(H)$ and then $v_av_b\in E(H)$.
Thus $\{v_1,\ldots,v_{c'+1}\}$ forms a clique of size $c'+1$,
a contradiction. Hence $W$ is an independent set.
By Claim~\ref{claim8}, the connectivity of $H$ and maximality of $c'$,
we have
$
N_H(w_i)=\{v_1,\ldots,v_{d_H(w_i)}\}$ and $
1\le d_H(w_i)\le c'-1
$
for every $w_i\in W$. Consequently,
\[
m-3=\binom{c'}2+\sum_{i=1}^{s}d_H(w_i)
=\binom d2+t.\ (1\le t\le d-1)
\]
 Define 
\[ 
Q_{c'}=\lambda^* I_{c'}-A(K_{c'})-\frac2{a_0} E_{11} 
=a_0I_{c'}-J_{c'}-\frac2{a_0}E_{11}. 
\] We first show that $Q_{c'}$ is always positive definite.
If $c'<d$, then $\lambda^*>d-1\ge2$ by Lemma~\ref{lemma2.7}(ii), so $\frac2{a_0}=\frac2{\lambda^*+1}<1$ and every row sum of $A(K_{c'})+\frac2{a_0} E_{11}$ is smaller than $c'$. Hence
$
\rho(A(K_{c'})+\frac2{a_0} E_{11})<c'\le d-1<\lambda^*
$, and $Q_{c'}$ is positive definite. If $c'=d$, then $\rho(A(K_d)+\frac2{a_0} E_{11})<\lambda^*$ by the proof of Lemma~\ref{lemma2.7}(ii), so $Q_d$ is also positive definite. 
For $1\le z\le c'-1$, let 	$\mathbf v_z=(\underbrace{1,\ldots,1}_{z},0,\ldots,0)^T$ $\in\mathbb R^{c'}$. Set $g_{c'}(z):=\mathbf v_z^T Q_{c'}^{-1} \mathbf v_z$ and set $g_{c'}(0)=0$. Recall that $m^*=a_0^2-2$ and $h=a_0-d$. Set $P_{c'}=a_0I_{c'}-J_{c'}$ and $h^*=a_0\lambda^*h-t(h+t)>0$. And by equation (4), $
m^*h^*=2(a_0\lambda^*+h+t)>0.$ Then 
\begin{align}
	Q_{c'}^{-1}
	&= P_{c'}^{-1}+\frac{\frac{2}{a_0}P_{c'}^{-1}e_1e_1^T P_{c'}^{-1}}{1-\frac{2}{a_0}e_1^T P_{c'}^{-1}e_1}\notag \\
	&= \frac{1}{a_0}I_{c'}+\frac{1}{a_0 (a_0-c')}J_{c'}+\frac{2((a_0-c')e_1+\mathbf 1)((a_0-c')e_1+\mathbf 1)^T}{a_0 (a_0-c')((a_0-c')m^*-2)},
	\tag{13}
\end{align}
and by equation (13),
\begin{align}
	g_{c'}(z) &= \mathbf{v}_z^T Q_{c'}^{-1}\mathbf{v}_z
	= \mathbf{v}_z^T P_{c'}^{-1}\mathbf{v}_z
	+\frac{2\big(\mathbf{v}_z^T\big((a_0-c'){e}_1+\mathbf{1}\big)\big)^2}{a_0(a_0-c')\big((a_0-c')m^*-2\big)} \notag \\
	&= \frac{z(a_0-c'+z)}{a_0(a_0-c')}
	+\frac{2(a_0-c'+z)^2}{a_0(a_0-c')\big((a_0-c')m^*-2\big)} \notag \\
	&= \frac{a_0-c'+z}{a_0(a_0-c')}
	\left(z+\frac{2(a_0-c'+z)}{(a_0-c')m^*-2}\right) \notag \\
	&= \frac{(a_0-c'+z)(z m^*+2)}{a_0\big((a_0-c')m^*-2\big)}.
	\tag{14} \label{eq:14}
\end{align}
Let $B_0$ be the $c'\times s$ matrix whose $i$-th column is
$\mathbf v_{d_H(w_i)}$. Then
\[
\lambda^* I-\left(A(H)+\frac2{a_0}E_{uu}\right)
=\begin{pmatrix}
	Q_{c'}&-B_0\\
	-B_0^T&\lambda^* I_s
\end{pmatrix}.
\]
Because $\lambda^* I_s$ is positive definite, this matrix is positive definite if and only if its Schur complement
$
Q_{c'}-\frac1{\lambda^*}B_0B_0^T
=Q_{c'}-\frac1{\lambda^*}
\sum_{i=1}^{s}\mathbf v_{d_H(w_i)}\mathbf v_{d_H(w_i)}^T
$
is positive definite.
By a congruence transformation with $Q_{c'}^{-1/2}$, we obtain
\[
\begin{aligned}
	&Q_{c'}^{-1/2}
	\left(Q_{c'}-\frac1{\lambda^*}B_0B_0^T\right)Q_{c'}^{-1/2}\\
	&\qquad=I_{c'}-\frac1{\lambda^*}Q_{c'}^{-1/2}
	\left(\sum_{i=1}^{s}\mathbf v_{d_H(w_i)}
	\mathbf v_{d_H(w_i)}^T\right)Q_{c'}^{-1/2}
	=I_{c'}-Z_{c'},
\end{aligned}
\]
where
$
Z_{c'}=\frac1{\lambda^*}Q_{c'}^{-1/2}
\left(\sum_{i=1}^{s}\mathbf v_{d_H(w_i)}
\mathbf v_{d_H(w_i)}^T\right)Q_{c'}^{-1/2}
$.
Since congruence transformations preserve positive definiteness, the positive definiteness of $Q_{c'}-\frac1{\lambda^*}B_0B_0^T$ is equivalent to that of $I_{c'}-Z_{c'}$,
 and the matrix $I_{c'}-Z_{c'}$ is positive definite if and only if $\rho(Z_{c'})<1$. Note that $Z_{c'}$ is positive semidefinite, then
 
\[	\rho(Z_{c'})
	\le\operatorname{tr}Z_{c'}
	=\frac1{\lambda^*}\sum_{i=1}^{s}
	\mathbf v_{d_H(w_i)}^TQ_{c'}^{-1}\mathbf v_{d_H(w_i)}
	=\frac1{\lambda^*}\sum_{i=1}^{s}g_{c'}\bigl(d_H(w_i)\bigr).\]
Accordingly,
$\rho(A(H)+\frac2{a_0}E_{uu})<\lambda^*$ follows provided
\[
\sum_{i=1}^{s}g_{c'}\bigl(d_H(w_i)\bigr)<\lambda^*.
\]
Since $Q_{c'}$ is positive definite, we have $a_0\big((a_0-c')m^*-2)>0$. By equation (14),  we obtain
$
g_{c'}''(z)=
\frac{2m^*}{a_0((a_0-c')m^*-2)}>0.
$
Thus $g_{c'}$ is strictly convex on $[1,c'-1]$.
We now determine the graph structure that maximizes
$\sum_{i=1}^{s}g_{c'}(d_H(w_i))$ among graphs with the above clique and neighborhood structure with $c'$ and $|E(H)| = m - 3$ fixed and with no isolated vertices.
Let $w_i,w_j$ be two distinct vertices of $W$.
We establish the following two key facts.

\begin{itemize}
	\item Fact 1: If $d_H(w_i)+d_H(w_j)\le c'-1$, then	$
	g_{c'}\bigl(d_H(w_i)+d_H(w_j)\bigr)
	>
	g_{c'}\bigl(d_H(w_i)\bigr)+g_{c'}\bigl(d_H(w_j)\bigr).
	$
\noindent\emph{Proof.}
	\[
	\begin{aligned}
		&g_{c'}\bigl(d_H(w_i)+d_H(w_j)\bigr)
		-g_{c'}\bigl(d_H(w_i)\bigr)-g_{c'}\bigl(d_H(w_j)\bigr)\\
		&\qquad=
		\frac{2m^*d_H(w_i)d_H(w_j)-2(a_0-c')}
		{a_0((a_0-c')m^*-2)}>0,
	\end{aligned}
	\]
	since $m^*d_H(w_i)d_H(w_j)-(a_0-c')\ge m^*-(a_0-c')=(\lambda^*)^2+\lambda^*+c'-2>0$.
	\item Fact 2: If $2\le d_H(w_i)\le d_H(w_j)<c'-1$ and
	$d_H(w_i)+d_H(w_j)>c'-1$, then	$
	g_{c'}\bigl(d_H(w_i)-1\bigr)+g_{c'}\bigl(d_H(w_j)+1\bigr)
	>
	g_{c'}\bigl(d_H(w_i)\bigr)+g_{c'}\bigl(d_H(w_j)\bigr).
	$

	\noindent\emph{Proof.}
	\[
	\begin{aligned}
		&g_{c'}\bigl(d_H(w_i)-1\bigr)+g_{c'}\bigl(d_H(w_j)+1\bigr)
		-g_{c'}\bigl(d_H(w_i)\bigr)-g_{c'}\bigl(d_H(w_j)\bigr)\\
		&\qquad=
		\frac{2m^*\bigl(d_H(w_j)-d_H(w_i)+1\bigr)}
		{a_0((a_0-c')m^*-2)}>0.
	\end{aligned}
	\]
\end{itemize}

Choose a graph $H'$ attaining this maximum among the graphs
under consideration. Suppose that $H'$ has two vertices
$w_i,w_j$ outside the clique such that
$
1\le d_{H'}(w_i)\le d_{H'}(w_j)<c'-1.
$
If $d_{H'}(w_i)+d_{H'}(w_j)\le c'-1$, we construct a new signed graph $H'_1$ by deleting the vertex $w_i$
and replacing the neighborhood of the vertex $w_j$ by
$
\{v_1,\ldots,v_{d_{H'}(w_i)+d_{H'}(w_j)}\}.
$
Obviously, $H'_1$ and $H'$ have the same number of edges and the same clique order.
 By Fact 1, its sum of
$g_{c'}$-values is strictly larger, a contradiction.
If $d_{H'}(w_i)+d_{H'}(w_j)>c'-1$, then
$d_{H'}(w_i)\ge2$. We construct a new signed graph $H'_2$ by deleting the positive edge
$w_iv_{d_{H'}(w_i)}$ and adding the positive edge
$w_jv_{d_{H'}(w_j)+1}$.
Obviously, $H'_2$ and $H'$ have the same number of edges
and the same clique order.
By Fact 2, its sum of $g_{c'}$-values is strictly
larger, a contradiction.
Therefore, all vertices outside the clique except
possibly one have degree $c'-1$ and are adjacent precisely
to $v_1,\ldots,v_{c'-1}$ in $H'$.
If there is an exceptional vertex $w$, then
$
1\le d_{H'}(w)=r<c'-1$ and $
N_{H'}(w)=\{v_1,\ldots,v_{d_{H'}(w)}\}.
$ Let $p$ be the number of vertices of $H'$ outside the clique
that have degree $c'-1$. Since $H'$ has $m-3$ edges, we have
$R=\sum_{i=1}^{s}d_H(w_i)=m-3-\binom{c'}2=p(c'-1)+r$,
where $0\le r<c'-1$. Consequently,
\[
\sum_{i=1}^{s}g_{c'}\bigl(d_H(w_i)\bigr)
\le p\,g_{c'}(c'-1)+\mathbf1_{\{r>0\}}g_{c'}(r),
\tag{15}
\]
with equality if and only if $H$ also has exactly $p$ vertices
outside the clique of degree $c'-1$ and, if $r>0$, one further
vertex of degree $r$, with no other vertices outside the clique.

If $2\le c'\le d-1$, then $|W|\ge 2$. Otherwise, we first show that $|W|\neq 0$. If $|W|=0$, then $H=K_{c'}$ and $m=\binom{c'}{2}+3\le \binom{d}{2}+3<\binom{d}{2}+t+3=m$, a contradiction. Write $W=\{w_1\}$, then
$	\binom d2+t
=m-3=\binom{c'}2+d_H(w_1)
\le\binom{c'}2+c'-1
=\binom{c'+1}2-1\le\binom d2-1,$
a contradiction. 
We claim that 
\[
pg_{c'}(c'-1)+\mathbf{1}_{\{r>0\}}g_{c'}(r)
=p\cdot\frac{\lambda^*\big((c'-1)m^*+2\big)}{a_0\big((a_0-c')m^*-2\big)}
+\mathbf{1}_{\{r>0\}}\cdot\frac{(a_0-c'+r)(rm^*+2)}{a_0\big((a_0-c')m^*-2\big)}
<\lambda^*.
\]
 Thus the above inequality is equivalent to $N>0$, where
\begin{equation}
	N = a_0\lambda^*\big((a_0-c')m^*-2\big)-p\lambda^*\big((c'-1)m^*+2\big)-\mathbf{1}_{\{r>0\}}(a_0-c'+r)(rm^*+2).
	\tag{16}\label{eq:N-def}
\end{equation}
Note that
\[
R=\binom d2+t-\binom{c'}{2}
=\frac{(d-c')(d+c'-1)}2+t
=(d-c'+1)(c'-1)+\widetilde M,
\tag{17}
\]
where
\[
\widetilde M=t-d+\frac{(d-c'+1)(d-c'+2)}2.
\tag{18}
\]
We distinguish the following three cases based on  $\widetilde M$. 
\setcounter{casee}{0} 

\begin{casee} $\widetilde M=0$. 
	
	By equation (17) and the definitions of $p,r$, we have
	$p=d-c'+1$ and $r=0$. Recall that $h^*=a_0\lambda^*h-t(h+t)$.
	By equations (16) and (4), we obtain
	\[
	\frac{h^* N}{2} 
	=\frac{(d-c')\lambda^*(d-c'+1)(d-c'+2)}4 
	\bigl((d-c'-1)\lambda^*+(d-c')(d-c'+3)\bigr)>0. 
	\] Thus $N>0$. 
\end{casee} 
\begin{casee} $\widetilde M=-j<0$.
	
	Since 
	$(c'-1)-j
	=c'-1+\widetilde M
	=t+\frac{(d-c')(d-c'+1)}2>0$, $1\le j<c'-1$.
	By equation (17) and the definitions of $p,r$, we have $p=d-c'$ and $r=c'-1-j$. Then
	\[
	\frac{h^* N}{2}
	-\frac{(d-c')\lambda^*(d-c'+1)(d-c'+2)}4
	\bigl((d-c'-1)\lambda^*+(d-c')(d-c'+3)\bigr)=-\frac{j(d-c')}{4}N_1,
	\]
	where 
	\[ 
	\begin{aligned} 
		N_1 
		&={}-4d^2(d-c')-4dj-4d(d-c')^2-8d(d-c')h-8d(d-c')-8d 
		+2j(d-c')-4jh+6j\\&+(d-c')^3-4(d-c')^2h+8(d-c')^2-4(d-c')h^2-8(d-c')h+15(d-c')-8h+8\\ 
		&<-4d^2(d-c')-4d(d-c')^2-8d(d-c')-8d+(d-c')^3+8(d-c')^2+15(d-c')+8\\ 
		&\le-(d-c'+1)(7(d-c')^2+17(d-c')+8)<0.\end{aligned} 
	\] 
	Thus $\frac{h^* N}{2}>0$, and hence $N>0$. 
\end{casee} 

\begin{casee} $\widetilde M>0$.
\end{casee}
Let $n_0$ be the number of vertices of $H'$ outside the clique.

Then
$
n_0=\left\lceil\frac{R}{c'-1}\right\rceil
=\begin{cases}p,&r=0,\\p+1,&1\le r\le c'-2.\end{cases}
$
Define
$
\beta=m^*\bigl(a_0(a_0-c')-R\bigr)-2(a_0+n_0).
$
If $r=0$, 
then $R=p(c'-1)$, $n_0=p$, and $\mathbf1_{\{r>0\}}=0$. By equation (16),
\[ 
\begin{aligned} 
	N &= a_0\lambda^*((a_0-c')m^*-2)-p\lambda^*((c'-1)m^*+2)\\ 
	&=\lambda^*\Big[ a_0((a_0-c')m^*-2)-p((c'-1)m^*+2) \Big]\\ 
	&=\lambda^*\Big[m^*(a_0(a_0-c')-p(c'-1))-2a_0-2p\Big]\\ 
	&=\lambda^*\Big[m^*\big(a_0(a_0-c')-R\big)-2(a_0+n_0)\Big]\\&=\lambda^*\beta. 
\end{aligned} 
\] 
If $r>0$,
then $R=p(c'-1)+r$, $n_0=p+1$, and $\mathbf1_{\{r>0\}}=1$. By $\lambda^*=a_0-1$ and $r<c'-1$, we have 
\[ 
\begin{aligned} 
	N-\lambda^*\beta 
	&=\Big[a_0\lambda^*((a_0-c')m^*-2)-p\lambda^*((c'-1)m^*+2)-(a_0-c'+r)(rm^*+2)\Big] \\ 
	&\quad -\Big[\lambda^*m^*\big(a_0(a_0-c')-R\big)-2\lambda^*(a_0+p+1)\Big] \\ 
	&=\lambda^*\big[m^*(a_0(a_0-c')-p(c'-1))-2a_0-2p\big]-(a_0-c'+r)(rm^*+2) \\ 
	&\quad -\lambda^*m^*(a_0(a_0-c')-p(c'-1)-r)+2\lambda^*(a_0+p+1) \\ 
	&= \lambda^*m^*r-(a_0-c'+r)(rm^*+2)+2\lambda^*\\ 
	&=(c'-1-r)(rm^*+2)\\
	&>0.
\end{aligned} 
\] 
Consequently, 
\[
N=\lambda^*\beta+
\mathbf1_{\{r>0\}}(c'-1-r)(rm^*+2).
\tag{19}
\]
Hence $N\ge\lambda^*\beta$.
Thus, to obtain $N>0$, it suffices to prove $\beta>0$.

By equations (17) and (18), we obtain
$
1\le d-t\le\frac{(d-c')(d-c'+3)}2
$
and
$
n_0=d-c'+1+
\left\lceil\frac{\widetilde M}{c'-1}\right\rceil
\le d-c'+1+\widetilde M.
$
By equation (4), together with $h=a_0-d$ and $\lambda^*=a_0-1$, we have
$
m^*\bigl((a_0^2-d+t)h-(a_0-d+t)^2\bigr)
=2(a_0^2-d+t).
$
Since $a_0>d$ and $1\le d-t\le d-1$,  $a_0^2-d+t>0$.
Hence
\[
m^*h=2+\frac{m^*(a_0-d+t)^2}{a_0^2-d+t}.
\tag{20}
\]
Also, by equation (17),
\[
\begin{aligned}
	a_0(a_0-c')-R
	&=(a_0+d-c')h+\frac{(d-c')(d-c'+1)}2-t\\
	&=(a_0+d-c'+1)h
	+\frac{(d-c')(d-c'+1)}2+d-t-a_0.
\end{aligned}
\]
Define
\[
\begin{aligned}
	\varphi(x)
	&=(x+d-c'+1)(x-d+t)^2+(x^2-d+t)
	\left(\frac{(d-c')(d-c'+1)}2+d-t-x\right)\\
	&=\widetilde Mx^2+(d-t)(2c'-d-t-1)x
	+\frac{(d-c')(d-t)(d+c'-2t-1)}2.
\end{aligned}
\]
By equation (20),
\begin{align}
	\beta
	&=(a_0+d-c'+1)m^*h
	+m^*\left(\frac{(d-c')(d-c'+1)}2+d-t-a_0\right)
	-2(a_0+n_0)\notag\\
	&=\frac{m^*}{a_0^2-d+t}\varphi(a_0)-2(n_0-d+c'-1).
	\tag{21}
\end{align}

Further analysis splits into the following two subcases.
\renewcommand{\thesubcasee}{3.\arabic{subcasee}}
\renewcommand{\theHsubcasee}{3.\arabic{subcasee}}
\setcounter{subcasee}{0}

\begin{subcasee} $d-c'\ge2$.
\end{subcasee}
Since $\varphi''(x)=2\widetilde M>0$ and
$\varphi'(d-t+1)
=(d-t)\bigl((d-c')^2+d-c'-1-(d-t)\bigr) +(d-c'+1)(d-c'+2),$
we have
\[
\begin{aligned}
	(d-c')^2+d-c'-1-(d-t)
	&\ge(d-c')^2+d-c'-1-\frac{(d-c')(d-c'+3)}2\\
	&=\frac{(d-c'-2)(d-c'+1)}2\ge0.
\end{aligned}
\]
Thus $\varphi'(d-t+1)>0$ and $\varphi$ is strictly increasing on
$[d-t+1,\infty)$. Since $a_0>d\ge d-t+1$, we obtain
\begin{align}
	\varphi(a_0)&>\varphi(d-t+1)\notag\\
	&=\frac12\Bigl[(d-c'-1)(d-c'+2)(d-t)^2\notag\\
	&+(d-c')(d-c'+1)(d-t)
	+(d-c'+1)(d-c'+2)\Bigr]>0.
	\tag{22}
\end{align}
Since $a_0>d\ge4$, we have
\[
\frac{m^*}{a_0^2-d+t}
=\frac{a_0^2-2}{a_0^2-(d-t)}
\ge\frac{a_0^2-2}{a_0^2-1}
>\frac{14}{15}.
\]
By equations (21) and (22), together with
$n_0-d+c'-1\le\widetilde M$, we obtain
\[
\begin{aligned}
	\beta
	&>\frac{14}{15}\varphi(d-t+1)-2\widetilde M\\
	&=\frac1{15}\Bigl[
	7(d-c'-1)(d-c'+2)(d-t-1)^2\\
	\qquad &+\bigl(21(d-c')^2+21(d-c')+2\bigr)(d-t-1)\\
	\qquad &+2(d-c')\bigl(3(d-c')-5\bigr)\Bigr]>0.
	(\text{since}\ d-c'\ge2\ \text{and}\  d-t\ge1)
\end{aligned}
\]

\begin{subcasee} $d-c'=1$.
\end{subcasee}
Then $d\ge3$ and $\widetilde M=t-d+3>0$, so either $t=d-2$ or $t=d-1$.
If $t=d-2$, then $R=2d-3$ and
$n_0=\lceil(2d-3)/(d-2)\rceil=3$.
Since $h+t=a_0-2$ and $a_0^2-d+t=m^*$, by equation (20), we have
$
m^*h=(a_0-2)^2+2=a_0^2-4a_0+6.
$
Since $d=a_0-h$, we have
\[
\begin{aligned}
	\beta
	&=m^*\bigl(a_0(h+1)-(2d-3)\bigr)-2(a_0+3)\\
	&=(a_0+2)m^*h+(3-a_0)m^*-2(a_0+3)\\
	&=(a_0+2)(a_0^2-4a_0+6)
	+(3-a_0)(a_0^2-2)-2(a_0+3)\\
	&=a_0(a_0-2)>0.
\end{aligned}
\]
If $t=d-1$, then $R=2d-2$ and
$n_0=\lceil(2d-2)/(d-2)\rceil\le4$.
Since $h+t=a_0-1=\lambda^*$,  by equation (4), we have
$
m^*\lambda^*(a_0h-t)=2\lambda^*(a_0+1).
$
Since $\lambda^*>0$ and $t=a_0-h-1$, we have
\[
\begin{aligned}
	\beta
	&=m^*\bigl(a_0(h+1)-(2d-2)\bigr)-2(a_0+n_0)\\
	&=(a_0+2)m^*h+(2-a_0)m^*-2(a_0+n_0)\\
	&=\frac{(a_0+2)(a_0^3-a_0^2+4)}{a_0+1}
	+(2-a_0)(a_0^2-2)-2(a_0+n_0)\\
	&=2a_0^2-2a_0-2n_0+2+\frac2{a_0+1}\\
	&\ge2a_0^2-2a_0-6+\frac2{a_0+1}>0.
	(\text{since}\ a_0>d\ge3)
\end{aligned}
\]
Thus $\beta>0$ in both subcases, and $N>0$ by equation (19) in Case 3.

In all three cases, we obtain
$
\sum_{i=1}^{s}g_{c'}\bigl(d_H(w_i)\bigr)
\le p\,g_{c'}(c'-1)+\mathbf1_{\{r>0\}}g_{c'}(r)<\lambda^*
$
whenever $2\le c'<d$, which implies
$
\rho\!\left(A(H)+\frac2{a_0}E_{uu}\right)<\lambda^*,
$
contrary to
$\rho(A(H)+\frac2{a_0}E_{uu})\ge\lambda\ge\lambda^*$.
Thus $c'=d$ and
$\sum_{i=1}^{s}d_H(w_i)=t$, where $1\le t\le d-1$.
Suppose for contradiction that $s\ge2$. Then equality cannot hold in (15), and hence
$
\sum_{i=1}^{s}g_d\bigl(d_H(w_i)\bigr)
<g_d(t)=\lambda^*,
$
where the last equality follows from the proof of
Lemma~\ref{lemma2.7}(ii).
By the preceding Schur-complement comparison, we obtain
$
\rho\!\left(A(H)+\frac2{a_0}E_{uu}\right)<\lambda^*
$,
contrary to
$\rho(A(H)+\frac2{a_0}E_{uu})\ge\lambda\ge\lambda^*$.
Therefore $s=1$, and
$d_H(w_1)=m-3-\binom d2=t$.
By Claim~\ref{claim8}, the unique vertex $w_1$ outside the clique is adjacent
precisely to $v_1,\ldots,v_t$. Since $v_1=u$ and $t\ge1$, it is adjacent to
$u$. Thus $H\cong H_{d,t}$. This completes the proof of Theorem~\ref{thm2}.

\section{Acknowledgement}
AI-assisted tools were employed for language refinement of the manuscript and to help verify computational results. All contents of this paper are the sole responsibility of the authors.

\begingroup\small\setlength{\itemsep}{2pt}\setlength{\parskip}{0pt}\endgroup

\end{document}